\documentclass[a4paper,11pt]{article}

\usepackage{amsmath}
\usepackage{amssymb}
\usepackage{tabularx}
\usepackage{enumerate}
\usepackage{setspace}
\usepackage{graphicx}
\usepackage{color}
\usepackage{graphicx}
\usepackage{caption}
\usepackage{subfigure}
\usepackage{bm}
\usepackage{float}

\usepackage[pagewise]{lineno}
\newtheorem{theorem}{Theorem}[section]

\newtheorem{lemma}{Lemma}[section]
\newtheorem{remark}{Remark}[section]

\newenvironment {proof} {\noindent {\bf Proof.}}{\quad $\blacksquare$\par\vspace{3mm}}

\newcommand{\be}{\begin{equation}}
\newcommand{\ee}{\end{equation}}
\newcommand\bes{\begin{eqnarray}}
\newcommand\ees{\end{eqnarray}}
\newcommand{\bess}{\begin{eqnarray*}}
\newcommand{\eess}{\end{eqnarray*}}

\begin{document}
\setlength{\baselineskip}{15.2pt} \pagestyle{myheadings}

\title{ \bf  \LARGE Wellposedness of solution for an $N$-D chemotaxis-convection model during tumor angiogenesis}
\date{\empty}
\author{
\sl{Fengxiang Zhao $^1$, \ Jiashan Zheng $^2$,\  Kaiqiang Li\ $^{3}$\thanks{Corresponding author.}}\\
{ \normalsize School of Mathematical and Informational Sciences, Yantai University}\\
{ \normalsize Yantai, 264005, Shandong, P. R. China}\\
{}\\
}
\footnotetext[1]{Email address: zfx2037@163.com}
\footnotetext[2]{Email address: zhengjiashan2008@163.com}
\footnotetext[3]{Email address: kaiqiangli19@163.com}
\maketitle

\begin{quote}
\noindent {\bf Abstract.} {In this paper, we consider the following parabolic-parabolic-elliptic system }
\begin{align*}
\left\{\aligned
& u_t=\Delta u-\nabla\cdot(u\nabla v)+\xi\nabla\cdot(u\nabla w)+au-\mu u^{\alpha}, && x\in\Omega, t>0,\\
& v_t=\Delta  v+\nabla\cdot(v\nabla w)-v+u,&& x\in\Omega, t>0,\\
& 0=\Delta w-w+u,&& x\in\Omega, t>0\\
\endaligned\right.
\end{align*}
on a bounded domain $\Omega\subset \mathbb{R}^{N}$ ($N\geq1$) with smooth boundary $\partial \Omega$, where $\mu$, $a$, $\alpha$ are positive constants and $\xi\in\mathbb{R}$. If one of the following cases holds:\\
(i) $N\geq4$ and $\alpha>\frac{4N-4+N\sqrt{2N^2-6N+8}}{2N}$;\\
(ii) $N=3$, $\alpha>2$, for any $\mu>0$ or $\alpha=2$, the index $\mu$ should be suitably big;\\
(iii) $N=2$, $\alpha\geq2$, for any $\mu>0$.\\
Without any restriction on the index $\xi$, for any given suitably regular initial data, the corresponding Neumann initial-boundary problem admits a unique global and bounded classical solution.

\noindent {\bf AMS subject classifications 2020:} {92C17, 35K55, 35K59, 35K20.}

\noindent {\bf Keywords:} {Tumor angiogenesis; classical solution; boundedness; chemotaxis-convection}
\end{quote}

\newcommand\HI{{\bf I}}
\section{Introduction and main result}
\setcounter{equation}{0}
In this paper, we consider a chemotaxis system, which is a simplified representation of the initial phase of tumor-related angiogenesis and provides insights into the complex biological processes of tumor-induced angiogenesis in the earliest stages.
\begin{equation}\label{origin}
\left\{\aligned
& u_t=\Delta u-\nabla\cdot(u\nabla v)+\xi\nabla\cdot(u\nabla w)+au-\mu u^{\alpha}, && x\in\Omega, t>0,\\
& v_t=\Delta v+\nabla\cdot(v\nabla w)-v+u,&& x\in\Omega, t>0,\\
& 0=\Delta w-w+u,&& x\in\Omega, t>0,\\
& \frac{\partial u}{\partial \nu}=\frac{\partial v}{\partial\nu}=\frac{\partial w}{\partial\nu}=0,&& x\in\partial\Omega, t>0,\\
& u(x,0)=u_0(x),\ v(x,0)=v_0(x),&& x\in\Omega
\endaligned\right.
\end{equation}
in a bounded domain $\Omega\subset\mathbb{R}^{N}$ $(N\geq1)$ with smooth boundary, where $\xi$, $\mu$ and $a$ are positive parameters, the unknown function $u(x, t)$ represents the density of endothelial cells, $v(x,t)$ is the density of adhesive sites, and $w(x,t)$ stands for the density of the matrix. Assuming that the quadruple of initial data $(u_0,v_0)$ satisfies
\begin{equation}\label{origin2}
\left\{\aligned
& u_{0} \in C^{0}(\overline{\Omega}), u_{0}(x) \neq 0, \\
& v_{0} \in W^{1,\infty}(\overline{\Omega}), v_{0}(x)>0.\\
\endaligned\right.
\end{equation}

Before introducing our main findings, we summarize several studies focusing on chemotaxis models stemming from tumor growth. Tumor is a vascular-dependent disease, showing growth, invasion and metastasis, all of which are closely related to angiogenesis. Tumor cells secrete a series of diffusible chemicals, which stimulate endothelial cells to migrate, rearrange into capillaries or buds and proliferate \cite{OMC 1996 JMAMB}. The growth of tumor consists of two different stages: avascular stage and vascular stage. During the avascular stage, the tumor lacks nutritional supplies and metabolic transport channels, resulting in its small size. Tumor cells mainly rely on diffusion to survive, which sometimes leads to dormancy. In the process of tumor growth, tumor cells secrete diffusive substances, particularly angiogenic factors, which is very important for understanding the regulation of vascular growth \cite{FK 1987science}. Endogenous negative regulators of angiogenesis, such as thrombospondin, angiostatin and glioma-derived angiogenesis inhibitory factor, play crucial roles in neovascularized tumours \cite{FA 1995 nature}. As the tumor progressed to its later stages, its exponential growth is facilitated by the penetration of capillaries into its interior, which providing direct nourishment to the malignancy \cite{FA 1950 A.C.R}. Angiogenesis, haptotaxis, and convection are all critical factors in the movement of endothelial cells, as evidenced by numerous studies \cite{PK 1989 CROH, SRGP 1985 IM, YO 1978 Nature}. Endothelial cells secrete a matrix, comprised of fibronectin, laminin, and type IV collagen, with the distribution of adhesive sites on this matrix influencing their migratory patterns. Diffusion, chemotaxis, and convection jointly control the kinetics of endothelial cell movement, as outlined in previous research \cite{Chaplain Stuart 1993 JMAMB, PK 1989 CROH, Stokes Lauffenburger 1991 JTB}. Scholars have delved into the fundamental principles and strategies for governing these processes through mathematical modeling \cite{Mcelwain Balding 1985 JTB, Edelstein 1982 JTB}. Additionally, there are a large number of studies exploring chemotaxis-only models and their variations, as referenced in \cite{Ding Wang 2019 DCDSSB, Fujie 2020 DCDSSS, Tian Zheng 2016 CPAA, Wu Shi 2017 CPAA, zhao 2020 DCDS}.

When considering the intricate branching and growth mode of capillary sprouts in the process of angiogenesis, many experts and scholars have delved into its fundamental principles and strategies through mathematical modeling. Notably, Orme and Chaplain \cite{OMC 1996 JMAMB} introduced a nondimensionalized chemotaxis-convection system, which provides a deep understanding of these complex biological phenomena to describe the branching of capillary sprout during tumor angiogenesis.
\begin{equation}\label{Chaplain}
\left\{\aligned
& u_{t}=d_{1}\Delta u-\nabla \cdot (u \nabla v)+\nabla \cdot (u \nabla w),&& x\in\Omega, t>0,\\
& v_{t}=d_{2}\Delta v +\nabla \cdot (v \nabla w)- \mu v+\kappa u,&& x\in\Omega, t>0,\\
& w_{t}=d_{3}\Delta w+ru-\delta w,&& x\in\Omega, t>0,\\
\endaligned\right.
\end{equation}
where $d_{1}$, $d_{2}$, $d_{3}$, $\kappa$, $r$, $\mu$ and $\delta$ are positive parameters, $u = u(x,t)$, $v=v(x,t)$ and $w=w(x,t)$ are the density of the endothelial cells, the adhesive sites and the matrix, respectively. The model \eqref{Chaplain} forms a critical coupling for the initial stage of tumor-related angiogenesis. As cells migrate upward, they will generate an adhesion gradient, which will be influenced by their interaction with the surrounding matrix. Especially, the convection of these cells heavily depends to the matrix. In addition, due to the convection term in the model, the spatial regularity of variable $v$ is closely related to the spatial regularity of variable $w$. This intricate coupling ensures the accuracy of the model in capturing the dynamic processes involved in tumor angiogenesis. In \cite{OMC 1996 JMAMB}, the authors demonstrated the model's abundant dynamic properties through the utilization of linear stability analysis and numerical simulation. The famous Keller-Segel system, which has garnered extensive research attention in various contexts \cite{ Tao Winkler 2015 SJMA, wink 2013 JMPA,  winkler 2010 JDE, winkler 2016 JMPA}, serves as a pertinent reference. With the expansion of the matrix, endothelial cells and adhesive sites are partially entrained within its expansion. It is noteworthy that the endothelial cells tend to migrate to areas with higher concentration of adhesion sites. Additionally, both the matrix and adhesive sites undergo a decay processes. Li and Tao \cite{Li Tao 2020 JMAA}  considered the chemotaxis-convection model \eqref{Chaplain} with $N=1$, and obtained that for all suitably regular initial data, an associated Neumann initial-boundary problem admits a globally defined bounded smooth solution.
In another study, Tao and Winkler \cite{Tao Winkler 2021 NA} concerned with a parabolic-parabolic-elliptic system, in $N$-dimensional bounded convex domains with $N\leq3$, a suitable assumption on largeness of the parameter corresponding to the latter repulsion term ensures global existence of a bounded classical solution.
Sun and Li \cite{sun li 2021 JMAA} considered the initial-boundary value problem and investigated the global bounded solution in the classical sense in two dimensional case,
\begin{equation}\label{Sun Chunlei}
\left\{\aligned
& u_{t}=d_{1}\Delta u-\nabla \cdot (u \nabla v)+\nabla \cdot (u \nabla w)+f(u),&& x\in\Omega, t>0,\\
& v_{t}=d_{2}\Delta v +\nabla \cdot (v \nabla w)- \mu v+\kappa u,&& x\in\Omega, t>0,\\
& w_{t}=d_{3}\Delta w+ru-\delta w,&& x\in\Omega, t>0.\\
\endaligned\right.
\end{equation}
Assuming that the function $f\in C^{1}([0,\infty))$ satisfies $f(s)\leq a-bs^{m}$ for $s>0$ with $a\geq0$ and $b>0$ as well as $f(0)\geq0$, when $m>3$, for all suitably regular initial data, the authors proved that the corresponding Neumann-type initial-boundary value problem possesses a globally defined bounded classical solution. The function $f$ adheres to specific growth constraints, thereby guaranteeing that endothelial cells undergo degradation and apoptosis at elevated population densities. It plays a crucial role in the regularity of solutions. Sun and Li \cite{sun li 2021 JMAA} noticed that on a large time scale, for the incentive example of  endothelial cells, individuals may die at high population density, then focused on the endothelial cells near the tip of the sprouts and considered the global bounded solution with $N=2$. Recently, Tang, Zheng and Li \cite{Tang} extended the previous global boundedness result. They considered a quasilinear attraction-repulsion chemotaxis system with logistic source and proved that there exists a unique global-in-time and bounded classical solution for all appropriately regular nonnegative initial data. Then, Zheng and Ke \cite{zheng ke 2022 CPAA} obtained the first result about the boundedness of solution for the system with higher dimensional $(N\geq3)$. When $f(u)=au-bu^{m}$, if $m>1+\frac{N}{2}$, the corresponding Neumann initial-boundary value problem admits a unique global bounded classical nonnegative solution. Moreover, the logistic source can prevent the blow up phenomenon.

Motivated by the aforementioned works, we delve into the parabolic-parabolic-elliptic system \eqref{origin}, disregarding any constraints on the spatial dimension and the index $\xi$. Indeed, our main result reduces the assumption on the exponent $\alpha$ in the logistic source compared to the above mentioned conditions, which ensured the boundedness results of the system \eqref{origin}. Our result also gives the relationship between the spatial dimensions with the index $\alpha$ in logistic terms.

Our main result can be described as follows:
\begin{theorem}\label{theorem} Let $N\geq1$ and assume that the initial data $u_{0}$ and $v_{0}$ satisfy \eqref{origin2} in a bounded domain with smooth boundary. If one of the following cases holds:

(i) $N\geq4$ and $\alpha>\frac{4N-4+N\sqrt{2N^2-6N+8}}{2N}$,

(ii) $N=3$, $\alpha>2$, for any $\mu>0$ or $\alpha=2$, the index $\mu$ should be suitably big,

(iii) $N=2$, $\alpha\geq2$, for any $\mu>0$,\\
then the problem \eqref{origin} admits a unique global bounded classical solution $(u, v, w)$ with the property that $u\geq0$, $v\geq0$ and $w\geq0$. Moreover, there exists $C>0$ such that
\begin{equation}\label{111}
\parallel v(\cdot,t)\parallel_{L^{\infty}(\Omega)}+\parallel u(\cdot,t) \parallel_{L^{\infty}(\Omega)}+\parallel w(\cdot,t)\parallel_{L^{\infty}(\Omega)}
\leq C \quad\mbox{for all}~~t>0.
\end{equation}
\end{theorem}

\begin{remark}
In this paper, we only give the proof of Theorem \ref{theorem} in the case $N=2$ and $\alpha=2$ for any $\mu>0$, while the case $N=2$ and $\alpha>2$ is standard which we omit the detail proof.
\end{remark}
\begin{remark}
Theorem \ref{theorem} significantly enhances the well-posedness results for the case of $N\geq 3$ and $\alpha>1+\frac{N}{2}$, which were obtained by Zheng and Ke \cite{zheng ke 2022 CPAA}. This advancement demonstrates a notable extension of the previous research, broadening its applicability and deepening our understanding of the relevant concepts.
\end{remark}
\begin{remark}
In this article, we demonstrate that, without any restrictions on the index $\xi$, for $N\geq1$, then \eqref{origin} has a global classical bounded solution. Furthermore, the research conducted in this paper extends and enhances the previously  results in \cite{Tao Winkler 2021 NA}.
\end{remark}
\begin{remark}
Our main result provides a substantial relaxation of the requirement about parameters and the space dimensions.
\end{remark}

Throughout this paper, we use symbols $C$ and $C_{i}$ $(i=1,2,\cdot\cdot\cdot)$ as some generic positive constants which may vary in the context.
Moreover, for simplicity, $u(x,t)$ is written as $u$, the integral $\int_{\Omega}u(x)dx$ is written as $\int_{\Omega}u(x)$ and $ \int^{t+\tau}_{t}\int_{\Omega} u(x,s)dxds$ is written as $ \int^{t+\tau}_{t}\int_{\Omega} u(x,s)$.

The rest of this paper is organized as follows.  In the following section, we state some preliminaries which will be used later, the local existence theorem also be described in Section 2. In Section 3, the a priori estimate of solution to \eqref{origin} are established for $N\geq4$. In Section 4, when $N=2$ and $N=3$, we give some main inequalities about the solution of \eqref{origin}. The proof of theorem \ref{theorem} is given in Section 5.

\newcommand\HII{{\bf II}}
\section{Local Existence and Preliminaries}
\setcounter{equation}{0}
In order to describe our main results, we first present several fundamental lemmas which will be used later. Notably, the Gagliardo-Nirenberg inequality will be invoked frequently throughout the forthcoming proofs.
\begin{lemma}\label{GN}(see \cite{ZZ2022JMAA})
Let $\Omega$ be a bounded Lipschitz domain in $\mathbb{R}^N$, $p$, $q$, $r$, $s$$\geq1$, $j$, $m\in\mathbb{N}_0$ and $\alpha\in[\frac{j}{m},1]$ satisfying
 $$\frac{1}{p}=\frac{j}{N}+\left(\frac{1}{r}-\frac{m}{N}\right)\alpha+\frac{1-\alpha}{q}.$$
 Then there are positive constants $C_{GN}$ and $G'_{GN}$ such that for all functions $w\in L^q(\Omega)$ with $\nabla w\in L^r(\Omega)$, $w\in L^s(\Omega)$,
 $$\|D^j w\|_{L^p(\Omega)}\leq C_{GN}\|D^m w\|_{L^r(\Omega)}^\alpha\|w\|_{L^q}^{1-\alpha}+C'_{GN}\|w\|_{L^s}.$$
\end{lemma}

We also give an ODE's theory, which plays a crucial role in our main proof.
\begin{lemma}\label{Le2}(see \cite{zheng  2015 JDE})
Let $y(t)\geq 0$ be a solution of problem
\begin{equation*}
\left\{\aligned
& y'(t)+Ay^{p} \leq B,\ \ \ t>0, \\
& y(0)=y_{0}\\
\endaligned\right.
\end{equation*}
with $ A>0$, $p>0$ and $B \geq 0$. Then for any $t>0$, we have
$$y(t)\leq \max \left\{y_{0},\left(\frac{B}{A}\right)^{\frac{1}{p}}\right\}.$$
\end{lemma}

\begin{lemma}\label{Le20}(see \cite{J2018JMAA})Let
\begin{equation*}
A_1=\frac{1}{\delta+1}\left[\frac{\delta+1}{\delta}\right]^{-\delta}\left(\frac{\delta-1}{\delta}\right)^{\delta+1}
\end{equation*}
and $ H(y)=y+A_1y^{-\delta}C_{\delta+1}$ for $y>0$. For any fixed $\delta>1$, $C_{\delta+1}>0$. Then
$$\min_{y>0}H(y)=\frac{\delta-1}{\delta}C^{\frac{1}{\delta+1}}_{\delta+1}.$$
\end{lemma}

Next, we recall the following elementary inequality.
\begin{lemma}\label{Le4}(see \cite{zheng ke 2022 CPAA})
Let $T>0$, $\tau\in(0,T)$, $A>0$, $\alpha>0$ and $B>0$, suppose that
$y: [0,T)\rightarrow[0,\infty)$ is absolutely continuous such that
\begin{equation*}
 y'(t)+Ay^{\alpha}(t) \leq h(t) \quad\mbox{for all}~~t\in(0,T) \\
\end{equation*}
with some nonnegative function $h \in L^{1}_{loc}([0,T))$ satisfying
\begin{equation*}
 \int^{t}_{t-\tau} h(s)ds \leq B \quad\mbox{for all}~~t\in(\tau,T). \\
 \end{equation*}
Then there exists a positive constant $C$ such that
\begin{equation*}
y(t)\leq max \{y_{0}+B,\frac{1}{\tau^{\frac{1}{\alpha}}}(\frac{B}{A})^{\frac{1}{\alpha}}+2B\} \quad\mbox{for all}~~t\in(\tau,T). \\
\end{equation*}
\end{lemma}

 Finally, we will give the following estimates of parabolic equation.
\begin{lemma}\label{Le7}(\cite{cao 2014 JMAA}, \cite{H P 1997 CPDE})
Suppose $\gamma\in(1,+\infty)$ and $g\in L^{\gamma}((0,T);L^{\gamma}(\Omega))$. Assuming $v$ is a solution of the following initial problem
\begin{equation*}
\left\{\aligned
& v_{t}-\Delta v +v=g, \\
& \frac{\partial v}{\partial \nu}=0,\\
& v(x,0)=v_{0}(x).
\endaligned\right.
\end{equation*}
Then there exists a positive constant $C_{\gamma}$ such that if $s_{0}\in[0,T)$, $v(\cdot,s_{0})\in W^{2,\gamma}(\Omega))$ with $\frac{\partial v(\cdot,s_{0})}{\partial \nu}=0$, then
\begin{align*}
& \int_{s_0}^Te^{\gamma s}\|\Delta v(\cdot,s)\|_{L^\gamma(\Omega)}^\gamma ds\\
 \leq& C_\gamma\left(\int_{s_0}^Te^{\gamma s}\|g(\cdot,s)\|_{L^\gamma(\Omega)}^\gamma ds+e^{\gamma s_0}(\|v_{0}(\cdot,s_0)\|_{L^\gamma(\Omega)}^\gamma+\|\Delta v_{0}(\cdot,s_0)\|_{L^\gamma(\Omega)}^\gamma)\right).
\end{align*}
\end{lemma}

Next, we review the well-known estimates for heat semigroup with homogeneous Neumann boundary conditions.
\begin{lemma}\label{semigroup}(See \cite{winkler 2010 JDE}, \cite{ZJ2021JDE})
Let $(e^{t\Delta})_{t}\geq0$ be the Neumann heat semigroup in $\Omega$
and let $\lambda_{1}>0$ denote the first nonzero eigenvalue of $-\Delta~in~ \Omega$ under Neumann boundary conditions.
Then there exists constant $C$,  if $1\leq q \leq p \leq \infty$, then
\begin{equation*}
 \|\nabla e^{t\Delta}w\|_{{L^{p}}(\Omega)}\leq C(1+t^{-\frac{n}{2}(\frac{1}{q}-\frac{1}{p})})e^{-\lambda_{1}t}\|w\|_{{L^{q}}(\Omega)}
\end{equation*}
for all $t>0$ and $w\in {L^{q}(\Omega)}$.
\end{lemma}

The following statement on local existence and extensibility can be verified by adapting arguments well-established in the context of related chemotaxis problems such as that detailed e.g. in \cite{WMZ2014JDE,WWW2012SIAM,Winkler2010CPDE}, to the present setting in a straightforward manner, so that we may refrain from presenting an elaborate proof here.
\begin{lemma}\label{Le3}
Let $\Omega \subset \mathbb{R}^{N}$ $(N \geq 1)$ be a smooth bounded domain with smooth boundary. Assume that \eqref{origin2} holds. Then there exist $T_{max}\in(0,\infty]$ and a uniquely determined triple of functions
\begin{equation}
\left\{\aligned
& u \in C^{0}(\overline{\Omega}\times [0,T_{max})\cap C^{2,1}(\overline\Omega\times(0,T_{max})),\\
& v \in \bigcap_{q>\max\{N,2\}}C^{0}([0,T_{max});W^{1,q}(\Omega))\cap C^{2,1}(\overline{\Omega}\times (0,T_{max})),\\
& w \in C^{2,0}(\overline{\Omega}\times (0,T_{max}))\\
\endaligned\right.
\end{equation}
such that $u,v,w$ are positive in $\overline\Omega\times(0,T_{max})$ and $(u,v,w)$ solves \eqref{origin} in the classical sense in $\Omega\times (0,T_{max})$, then if $T_{max}<\infty$,
$$\limsup\limits_{t\nearrow T_{max}}\left\{\|u(\cdot,t)\|_{L^\infty(\Omega)}+\|v(\cdot,t)\|_{L^{\infty}(\Omega)}+\|w(\cdot,t)\|_{W^{1,\infty}(\Omega)}\right\}
=\infty.$$
\end{lemma}

In view of Lemma \ref{Le3}, we have the following estimates.
\begin{lemma}\label{Le71}(see \cite{Zheng2018JMAA})
Let $\Omega\subset\mathbb{R}^{N}$ $(N\geq1)$ be a bounded domain with smooth boundary and $T_{max}\in(0,\infty)$. Then one can pick any $s_{0}\leq1$, for any $p>1$, there exists $K>0$ such that for all $\tau\in[0,s_{0}]$ satisfying
\begin{align}\label{72}
\parallel v(\cdot,\tau)\parallel_{L^{p}(\Omega)}+\parallel u(\cdot,\tau) \parallel_{L^{p}(\Omega)}+\parallel \Delta w(\cdot,\tau)\parallel_{L^{p}(\Omega)}\leq K.
\end{align}

\end{lemma}

\newcommand\HIII{{\bf III}}
\section{ A priori estimates for $N\geq4$}
\setcounter{equation}{0}
In this section, in what follows, we shall focus on deriving the a priori estimates of the local solution in order to extend it to be a global one. Henceforth, let $(u,v,w)$ be the classical solution of \eqref{origin} with the maximal existence time $T_{max}\in(0,\infty]$. Some basic properties can be derived as below. Let
\begin{equation}\label{L01}
\tau :=\mbox{min}\left\{1,\frac{1}{2}T_{max}\right\}.
\end{equation}

\begin{lemma}\label{Le3.1}
Let $\alpha>1$, under the assumptions in Lemma \ref{Le3}, then there exists $C>0$ such that the solution of \eqref {origin} satisfies
\begin{align}\label{Le30}
\int_{\Omega}u(\cdot,t) +\int_{\Omega}v(\cdot,t) \leqslant C \quad\mbox{for all}~~t\in(0,T_{max}).
\end{align}
\end{lemma}
\begin{proof}
Integrating by parts in the first equation of \eqref{origin}, one has
 \begin{align}\label{0L01}
\frac{d}{dt} \int_{\Omega} u =\int_{\Omega}(au-\mu u^{\alpha}) \quad\mbox{for all}~~t\in(0,T_{max}).
\end{align}
Combining with the Cauchy-Schwarz inequality can derive that
\begin{align}\label{0L001}
\frac{d}{dt} \int_{\Omega} u
\leq a\int_{\Omega}u-\frac{\mu}{|\Omega|^{\alpha-1}}\left(\int_{\Omega} u\right)^{\alpha}\quad\mbox{for all}~~t\in(0,T_{max}).
\end{align}
Since $\alpha>1$, using the Young's inequality, we can obtain
\begin{align}\label{0L0001}
\frac{d}{dt} \int_{\Omega} u 
\leq& \frac{\alpha-1}{\alpha}\left(\frac{\mu\alpha}{2|\Omega|^{\alpha-1}} \right)^{-\frac{1}{\alpha-1}}a^{\frac{\alpha}{\alpha-1}}-\frac{\mu}{2|\Omega|^{\alpha-1}}\left(\int_{\Omega} u\right)^{\alpha}\nonumber\\
=& \frac{\alpha-1}{\alpha} \left(\frac{2|\Omega|^{\alpha-1}}{\mu \alpha}\right)^{\frac{1}{\alpha-1}} a^{\frac{\alpha}{\alpha-1}}-\frac{\mu}{2|\Omega|^{\alpha-1}}\left(\int_{\Omega} u\right)^{\alpha}\quad\mbox{for all}~~t\in(0,T_{max}).
\end{align}
Moreover, integrating \eqref{0L0001} and using Lemma \ref{Le2}, one has
\begin{align}\label{L11}
& \int_{\Omega} u\leq \mbox{max} \left\{\int_{\Omega}u_{0},C_{1}\mu^{-\frac{1}{\alpha-1}}\right\} \quad\mbox{for all}~~t\in(0,T_{max}),
\end{align} \\
where $$C_{1}=(\alpha-1)^{\frac{1}{\alpha}}|\Omega|(\frac{2}{\alpha})^{\frac{1}{\alpha-1}}a^{\frac{1}{\alpha-1}}.$$
Integrating the second equation in \eqref{origin}, we can draw that
\begin{align}\label{L2}
&\frac{d}{dt} \int_{\Omega} u +\int_{\Omega}v=\int_{\Omega}u \quad\mbox{for all}~~t\in(0,T_{max}).
\end{align}
From the above estimate \eqref{L11}, we can see that there is a positive constant $C_{2}$ such that
\begin{align}\label{L21}
\int_{\Omega}v\leq C_{2} \quad\mbox{for all}~~t\in(0,T_{max}).
\end{align}
Combining \eqref{L11} and \eqref{L21} can obtain the estimate \eqref{Le30}. Hereto the proof is completed.
\end{proof}

By means of the first equation in \eqref{origin}, we can obtain that the spatiotemporal estimate of $u$.
\begin{lemma}\label{Le3.0}
Let $\alpha>1$. There exists a positive constant $C$ such that
\begin{equation}
\int^{t+\tau}_{t}\int_{\Omega}u^{\alpha}(\cdot,s)dxds\leq C \quad\mbox{for all}~~t\in(\tau,T_{max}),
\end{equation}
where $\tau$ is given by \eqref{L01}.
\end{lemma}
\begin{proof}
Integrating the first equation of \eqref{origin} can draw that
\begin{equation}\label{Le3.10}
\frac{d}{dt}\int_{\Omega}u=a\int_{\Omega}u-\mu\int_{\Omega}u^{\alpha}\quad\mbox{for all}~~t\in(0,T_{max}).
\end{equation}
Integrating \eqref{Le3.10} over $(t,t+\tau)$ and using \eqref{L11}, then we have
\begin{align}\label{Le3.11}
 \int^{t+\tau}_{t}\int_{\Omega}u^{\alpha}   
\leq &\frac{1}{\mu}\left[\int_{\Omega}u(\cdot,t+\tau)+a \tau ~\sup\limits_{s\in(t,t+\tau)}\int_{\Omega} u(\cdot,s
)\right]\nonumber\\
\leq &  \frac{a+1}{\mu} \int_{\Omega} u\nonumber\\
\leq& C_{1} \quad\mbox{for all}~~t\in(\tau,T_{max}),
\end{align}
where $\tau$ is given by \eqref{L01} and $C_1$ is a positive constant. In consequence, the proof of this lemma is finished.
\end{proof}

By means of \eqref{Le30} and using the equation of $w$, we can see that $w$ does possess the better property than $L^1$-boundedness.
\begin{lemma}\label{Le030.2}
Under the assumptions in lemma \ref{Le3}, then there exists a positive constant $C$ satisfying
\begin{equation}
\parallel w(\cdot,t)\parallel_{L^{l_0}(\Omega)}\leq C \quad\mbox{for all}~~t\in(0,T_{max}).
\end{equation}
where $l_0\in[1,\frac{N}{(N-2)_+})$.
\end{lemma}
\begin{proof}
 According to the third equation of \eqref{origin} and Lemma \ref{Le3.1}, there exists a constant $C_1>0$ such that $\int_{\Omega}u\leq C_1$, $t\in(0,T_{max})$. Therefore, invoking the classical result by Br$\acute{\mathrm{e}}$zis and Strauss \cite{HW1973JMSJ} as well as the Minkowski inequality, for any $l_1\in[1,\frac{N}{(N-1)_+})$, we conclude that there exist constants $C_2>0$ and $C_3>0$ as well as $C_4>0$ such that
\begin{align}\label{Le7.111221}
\parallel w\parallel_{W^{1,l_1}(\Omega)}\leq& C_2\parallel \Delta w-w\parallel_{L^{1}(\Omega)}\nonumber\\
\leq& C_3\parallel u\parallel_{L^{1}(\Omega)}\leq C_4\quad\mbox{for all}~~t\in(0,T_{max}).
\end{align}
This together with the Sobolev embedding theorem implies that
\begin{align}\label{Le7.111222}
&\parallel w\parallel_{L^{l_0}(\Omega)}\leq C_5\quad\mbox{for all}~~t\in(0,T_{max})~~\quad\mbox{and}~~l_0\in\left[1,\frac{N}{(N-2)_+}\right)
\end{align}
with a positive constant $C_5$.
\end{proof}

Next, we will give the estimates of $u$ in $L^l$ and $v$ in $L^{p_0}$ space by the standard testing procedures method, which will be used to establish existence result of the global and bounded classical solution.
\begin{lemma}\label{Le3.2}
Let $N\geq4$ and $\alpha>\frac{(N-2)(N+2)}{2N}$, for any $1<p_{0}<\frac{2\alpha}{N-2}-\frac{2}{N}$, then there exists constant $C>0$ such that
\begin{equation}
\parallel v(\cdot,t)\parallel_{L^{p_{0}}(\Omega)}\leq C \quad\mbox{for all}~~t\in(0,T_{max}).
\end{equation}
\end{lemma}
\begin{proof}
Multiplying both sides of the second equation in \eqref{origin} by $v^{p_0-1}$, using integrations by parts, then combining the third equation of system \eqref{origin}, we obtain that
\begin{align}\label{Lp0}
&\int_{\Omega}v^{p_0-1}v_{t}=\int_{\Omega}v^{p_0-1} \Delta v+\int_{\Omega}\nabla\cdot(v\nabla w)v^{p_0-1}-\int_{\Omega}v^{p_0}+\int_{\Omega}uv^{p_0-1}
\end{align}
for all $t\in(0,T_{max})$. Recalling the third equation of \eqref{origin}, we can derive that
\begin{align}\label{Lp}
& \frac{1}{p_0}\frac{d}{dt}\int_{\Omega}v^{p_0}+\left(p_0-1\right)\int_{\Omega}v^{p_0-2}|\nabla v|^{2}\nonumber\\
=&\frac{p_0-1}{p_0}\int_{\Omega}v^{p_0}(w-u)+\int_{\Omega}v^{p_0-1}u-\int_{\Omega}v^{p_0}\nonumber\\
=&\frac{p_0-1}{p_0}\int_{\Omega}v^{p_0}w-\frac{p_0-1}{p_0}\int_{\Omega}v^{p_0}u+\int_{\Omega}v^{p_0-1}u-\int_{\Omega}v^{p_0} \quad\mbox{for all}~~t\in(0,T_{max}).
\end{align}
By means of the Young's inequality, for any $\varepsilon>0$, there exists $C_{1}>0$ such that
\begin{align}\label{Lp1}
\frac{p_0-1}{p_0}\int_{\Omega}v^{p_0}w\leq &\varepsilon \int_{\Omega}v^{p_0+\frac{2}{N}}+C_{1}\int_{\Omega}w^{\frac{p_0+\frac{2}{N}}{\frac{2}{N}}}\nonumber\\
=&\varepsilon \int_{\Omega}v^{p_0+\frac{2}{N}}+C_{1}\int_{\Omega}w^{q}
 \quad\mbox{for all}~~t\in(0,T_{max}),
\end{align}
where $q=\frac{p_0+\frac{2}{N}}{\frac{2}{N}}<\frac{N\alpha}{N-2}$ by using the fact that $p_{0}<\frac{2\alpha}{N-2}-\frac{2}{N}$, which together with the Gagliardo-Nirenberg inequality and the $L^{p}$ theory of elliptic equation as well as Lemma \ref{Le030.2}, we have
\begin{align}\label{3.001}
C_{1}\parallel w\parallel^{q}_{L^{q}(\Omega)}
\leq& C_2(\parallel \Delta w\parallel^{q\xi_1}_{L^{\alpha}(\Omega)}\parallel w\parallel^{q(1-\xi_1)}_{L^{l_0}(\Omega)}+\parallel w\parallel^{q}_{L^{l_0}(\Omega)})\nonumber\\
\leq&C_3(\parallel \Delta w\parallel^{q\xi_1}_{L^{\alpha}(\Omega)}+1)\nonumber\\
\leq&C_4\parallel u\parallel^{q\xi_1}_{L^{\alpha}(\Omega)}+C_3  \quad\mbox{for all}~~t\in(0,T_{max})
\end{align}
with some constants $C_i > 0$ $(i=2,3,4)$ and
$$\xi_1=\frac{\frac{1}{l_0}-\frac{1}{q}}{\frac{2}{N}+\frac{1}{l_0}-\frac{1}{\alpha}}\in(0,1).$$
Here we have used the fact that $q\xi_1\leq\alpha$, where $l_0$ is the same as Lemma \ref{Le030.2}.
Subsequently, collecting \eqref{Lp1} and \eqref{3.001} can draw that
\begin{align}\label{3.002}
\frac{p_0-1}{p_0}\int_{\Omega}v^{p_0}w\leq &\varepsilon \int_{\Omega}v^{p_0+\frac{2}{N}}+C_{1}\int_{\Omega}w^{\frac{p_0+\frac{2}{N}}{\frac{2}{N}}}\nonumber\\
\leq&\varepsilon \int_{\Omega}v^{p_0+\frac{2}{N}}+C_5\parallel u\parallel^{\alpha}_{L^{\alpha}(\Omega)}+C_6 \quad\mbox{for all}~~t\in(0,T_{max})
\end{align}
with some $C_5>0$ and $C_6>0$. According to Lemma \ref{Le3.1}, using the Young's inequality, one has
\begin{align}\label{Lp2}
\int_{\Omega}v^{p_0-1}u\leq& \frac{p_0-1}{p_0}\int_{\Omega}uv^{p_0}+\frac{1}{p_0}\int_{\Omega}u\nonumber\\
 \leq& \frac{p_0-1}{p_0}\int_{\Omega}uv^{p_0}+C_{7} \quad\mbox{for all}~~t\in(0,T_{max}),
\end{align}
where $C_7>0$ is a constant. For the second term in the left hand side of \eqref{Lp}, we can derive that
\begin{align}\label{Lp3}
& (p_0-1)\int_{\Omega}v^{p_0-2}|\nabla v|^{2}=(p_0-1)\frac{4}{p_0^{2}}\parallel\nabla v^{\frac{p_0}{2}}\parallel^{2}_{L^{2}(\Omega)} \quad\mbox{for all}~~t\in(0,T_{max}).
\end{align}
Therefore, using the Gagliardo-Nirenberg interpolation inequality and Lemma \ref{Le3.1}, one can find $C_8>0$ and $C_9>0$ such that
\begin{align}\label{Lp4}
\int_{\Omega}v^{p_0+\frac{2}{N}}=&\int_{\Omega}v^{\frac{p_0}{2}\cdot\frac{2(p_0+\frac{2}{N})}{p_0}}\nonumber\\
=& \parallel v^{\frac{p_0}{2}} \parallel^{\frac{2(p_0+\frac{2}{N})}{p_0}}_{L^{\frac{2(p_0+\frac{2}{N})}{p_0}}(\Omega)}\nonumber\\
\leq& C_{8} \parallel |\nabla v|^{\frac{p_0}{2}} \parallel^{\frac{2(p_0+\frac{2}{N})}{p_0}\cdot \theta}_{{L^{2}(\Omega)}} \parallel v^{\frac{p_0}{2}} \parallel ^{\frac{2(p_0+\frac{2}{N})}{p_0}\cdot (1-\theta)}_{L^{\frac{2}{p_0}}(\Omega)}+C_{8} \parallel v^{\frac{p_0}{2}} \parallel ^{\frac{2(p_0+\frac{2}{N})}{p_0}}_{L^{\frac{2}{p_0}}(\Omega)}\nonumber\\
\leq& C_{9}\left(\parallel |\nabla v|^{\frac{p_0}{2}}\parallel^{\frac{2(p_0+\frac{2}{N})}{p_0}\cdot \theta}_{{L^{2}(\Omega)}}+1\right)\quad\mbox{for all}~~t\in(0,T_{max}),
\end{align}
where $\theta=\frac{\frac{Np_0}{2}-\frac{Np_0}{2(p_0+\frac{2}{N})}}{1-\frac{N}{2}+\frac{Np_0}{2}}$. In view of $\frac{2(p_0+\frac{2}{N})}{p_0}\cdot \theta<2$, with the help of Young's inequality, there exist positive constants $C_{10}$ and $C_{11}$ such that
\begin{align}\label{Lp49991}
\int_{\Omega}v^{p_0+\frac{2}{N}}\leq& C_{10}\parallel |\nabla v|^{\frac{p_0}{2}}\parallel^{2}_{{L^{2}(\Omega)}}+C_{11}\quad\mbox{for all}~~t\in(0,T_{max}).
\end{align}
Thus,
\begin{align}\label{Lp41}
\parallel|\nabla v|^{\frac{p_0}{2}}\parallel^{2}_{L^{2}}\geq \frac{1}{C_{10}}\int_{\Omega}v^{p_0+\frac{2}{N}}-\frac{C_{11}}{C_{10}}\quad\mbox{for all}~~t\in(0,T_{max}).
\end{align}
Combining \eqref{Lp3} with \eqref{Lp41} can show that
\begin{align}\label{Lp42}
\frac{p_0^{2}}{4}\int_{\Omega}v^{p_0-2}|\nabla v|^{2}\geq \frac{1}{C_{10}}\int_{\Omega}v^{p_0+\frac{2}{N}}-\frac{C_{11}}{C_{10}}\quad\mbox{for all}~~t\in(0,T_{max}).
\end{align}
Inserting the above inequality into \eqref{Lp}, we can find $C_{12}>0$ such that
\begin{align}\label{Lp43}
&\frac{1}{p_0}\frac{d}{dt}\int_{\Omega}v^{p_0}+\frac{\left(p_0-1\right)}{C_{10}}\frac{4}{p_0^{2}}
\int_{\Omega}v^{p_0+\frac{2}{N}}\nonumber\\
\leq&\varepsilon \int_{\Omega}v^{p_0+\frac{2}{N}}+C_5\int_{\Omega}u^{\alpha}-\frac{p_0-1}{p_0}\int_{\Omega}v^{p_0}u+\frac{p_0-1}{p_0}\int_{\Omega}
uv^{p_0}
-\int_{\Omega}v^{p_0}
+C_{12}\nonumber\\
=&\varepsilon \int_{\Omega}v^{p_0+\frac{2}{N}}+C_5\int_{\Omega}u^{\alpha}-\int_{\Omega}v^{p_0}+C_{12}
 \quad\mbox{for all}~~t\in(0,T_{max}).
\end{align}
Choosing $\varepsilon <\frac{(p-1)}{C_{10}}\frac{4}{p_0^{2}}$, then
\begin{align}\label{Lp5}
\frac{1}{p_0}\frac{d}{dt}\int_{\Omega}v^{p_0} + \int_{\Omega}v^{p_0}\leq C_5\int_{\Omega}u^{\alpha}+C_{12}\quad\mbox{for all}~~t\in(0,T_{max}).
\end{align}
Based on the result of Lemma \ref{Le3.0}, we can pick $C_{13}>0$ such that
\begin{equation*}
\int^{t+\tau}_{t}\int_{\Omega}u^{\alpha}\leq C_{13} \quad\mbox{for all}~~t\in(\tau,T_{max}).
\end{equation*}
With the help of Lemma \ref{Le4}, we can obtain the result of this lemma.
\end{proof}

As a basic step of a priori estimates, we establish the boundedness of $u$ for the case $N\geq4$.
\begin{lemma}\label{Le3.03}
Let $N\geq4$ and $\alpha>\frac{4N-4+N\sqrt{2N^2-6N+8}}{2N}$. Then for any $l>1$, there is a positive constant $C$ such that the solution of \eqref{origin} has the property
\begin{equation}\label{Le3.000311}
\parallel u(\cdot,t)\parallel_{L^{l}(\Omega)}\leq C \quad\mbox{for all}~~t\in(0,T_{max}).
\end{equation}
\end{lemma}
\begin{proof}
Let $\alpha>2$. Multiplying the first equation of \eqref{origin} by $u^{l-1}$ and integrating by parts over $\Omega$, we arrive at
\begin{align}\label{Le3.00031}
&\frac{1}{l}\frac{d}{dt}\int_{\Omega}u^{l}+(l-1)\int_{\Omega}u^{l-2}|\nabla u|^{2}\nonumber\\
=&-\frac{(l-1)}{l}\int_{\Omega}u^{l}\Delta v +\frac{\xi(l-1)}{l}\int_{\Omega}u^{l}\Delta w+ a\int_{\Omega}u^{l}-\mu\int_{\Omega}u^{l+\alpha-1}
\end{align}
for all $t\in(0,T_{max})$. Using the Young's inequality can draw that
\begin{align}\label{Le3.00032}
&\frac{1}{l}\frac{d}{dt}\int_{\Omega}u^{l}+(l-1)\int_{\Omega}u^{l-2}|\nabla u|^{2}\nonumber\\
\leq&-\frac{(l-1)}{l}\int_{\Omega}u^{l}\Delta v +\frac{\xi(l-1)}{l}\int_{\Omega}u^{l}\Delta w-\frac{\mu}{2}\int_{\Omega}u^{l+\alpha-1}+C_{1}
\end{align}
for all $t\in(0,T_{max})$ and $C_{1}>0$. For the second term on the right hand side of the above estimate, there exists a constant $C_{2}>0$ satisfying
\begin{align}\label{Le3.00033}
&\frac{\xi(l-1)}{l}\int_{\Omega}u^{l}\Delta w\nonumber\\
\leq& \frac{\mu}{4}\int_{\Omega}u^{l+\alpha-1}+C_{2}\int_{\Omega}|\Delta w|^{\frac{l+\alpha-1}{\alpha-1}} \quad\mbox{for all}~~t\in(0,T_{max}).
\end{align}
Similarly, applying the Young's inequality, if $\alpha>2$, for the first term on the right hand side of \eqref{Le3.00032}, we have
\begin{align}\label{Le3.00034}
-\frac{(l-1)}{l}\int_{\Omega}u^{l}\Delta v
\leq\frac{\mu}{8}\int_{\Omega}u^{l+\alpha-1}+C_{3}\int_{\Omega}|\Delta v|^{\frac{l+\alpha-1}{\alpha-1}} \quad\mbox{for all}~~t\in(0,T_{max}),
\end{align}
where $C_{3}>0$ is a constant. Combining \eqref{Le3.00032}--\eqref{Le3.00034}, we can draw that
\begin{align}\label{Le3.00035}
\frac{1}{l}\frac{d}{dt}\int_{\Omega}u^{l}
\leq -\frac{\mu}{8}\int_{\Omega}u^{l+\alpha-1} +C_{3}\int_{\Omega}|\Delta v|^{\frac{l+\alpha-1}{\alpha-1}}+C_{2}\int_{\Omega}|\Delta w|^{\frac{l+\alpha-1}{\alpha-1}}+C_{1}
\end{align}
for all $t\in(0,T_{max})$. Now, in view of the third equation in \eqref{origin}, we have
\begin{align}\label{Le30.00036}
&C_2\parallel \Delta w\parallel^{\frac{l+\alpha-1}{\alpha-1}}_{L^{\frac{l+\alpha-1}{\alpha-1}}(\Omega)}\nonumber\\
\leq& C_4\parallel  w\parallel^{\frac{l+\alpha-1}{\alpha-1}}_{L^{\frac{l+\alpha-1}{\alpha-1}}(\Omega)}+C_5\parallel  u\parallel^{\frac{l+\alpha-1}{\alpha-1}}_{L^{\frac{l+\alpha-1}{\alpha-1}}(\Omega)} \quad\mbox{for all}~~t\in(0,T_{max})
\end{align}
with $C_i>0$ $(i=4,5)$.
By utilizing the Gagliardo-Nirenberg inequality, Lemma \ref{Le030.2} and the Young's inequality, we can deduce that there are some constants $C_i>0$ $(i={6,7})$ such that
\begin{align}\label{Le30.00035}
&C_4\parallel w\parallel^{\frac{l+\alpha-1}{\alpha-1}}_{L^{\frac{l+\alpha-1}{\alpha-1}}(\Omega)}\nonumber\\
\leq& {C_6}(\parallel \Delta w\parallel^{\frac{l+\alpha-1}{\alpha-1}\xi_2}_{L^{\frac{l+\alpha-1}{\alpha-1}}(\Omega)}\parallel  w\parallel^{\frac{l+\alpha-1}{\alpha-1}(1-\xi_2)}_{L^{l_0}(\Omega)}+\parallel  w\parallel^{\frac{l+\alpha-1}{\alpha-1}}_{L^{l_0}(\Omega)})\nonumber\\
\leq& \frac{C_2}{2}\parallel \Delta w\parallel^{\frac{l+\alpha-1}{\alpha-1}}_{L^{\frac{l+\alpha-1}{\alpha-1}}(\Omega)}+{C_7} \quad\mbox{for all}~~t\in(0,T_{max}),
\end{align}
where $\xi_2=\frac{1-\frac{2}{N}-\frac{\alpha-1}{l+\alpha-1}}{1-\frac{\alpha-1}{l+\alpha-1}}\in(0,1)$. Here, $l_0$ is the same as Lemma \ref{Le030.2}. Combining
\eqref{Le30.00035} with \eqref{Le30.00036}, we have
\begin{align}\label{Le30.000300006}
&\parallel \Delta w\parallel^{\frac{l+\alpha-1}{\alpha-1}}_{L^{\frac{l+\alpha-1}{\alpha-1}}(\Omega)}\nonumber\\
\leq& \frac{C_2}{2} \parallel \Delta w\parallel^{\frac{l+\alpha-1}{\alpha-1}}_{L^{\frac{l+\alpha-1}{\alpha-1}}(\Omega)}+C_5\parallel  u\parallel^{\frac{l+\alpha-1}{\alpha-1}}_{L^{\frac{l+\alpha-1}{\alpha-1}}(\Omega)}+{C_{7}}\quad\mbox{for all}~~t\in(0,T_{max}),
\end{align}
which implies that
\begin{align}\label{Le30.00037}
C_2\parallel \Delta w\parallel^{\frac{l+\alpha-1}{\alpha-1}}_{L^{\frac{l+\alpha-1}{\alpha-1}}(\Omega)}\
\leq& {C_{8}}\parallel  u\parallel^{\frac{l+\alpha-1}{\alpha-1}}_{L^{\frac{l+\alpha-1}{\alpha-1}}(\Omega)}+{C_{9}}\quad\mbox{for all}~~t\in(0,T_{max})
\end{align}
with $C_i>0$ $(i={8,9})$.
Moreover, adding $\frac{l+\alpha-1}{l(\alpha-1)}\int_{\Omega}u^{l}$ on both sides of  \eqref{Le3.00035}. Since $\alpha>2$, by using the Young's inequality can derive that
\begin{align}\label{Le3.00038}
&\frac{1}{l}\frac{d}{dt}\int_{\Omega}u^{l} +\frac{l+\alpha-1}{l(\alpha-1)}\int_{\Omega}u^{l} \nonumber\\
\leq& -\frac{\mu}{8}\int_{\Omega}u^{l+\alpha-1}+\frac{l+\alpha-1}{l(\alpha-1)}\int_{\Omega}u^{l} +C_{2}\int_{\Omega}|\Delta w|^{\frac{l+\alpha-1}{\alpha-1}}+C_{3}\int_{\Omega}|\Delta v|^{\frac{l+\alpha-1}{\alpha-1}}+C_{1}\nonumber\\
\leq&-\frac{\mu}{16}\int_{\Omega}u^{l+\alpha-1}+C_{2}\int_{\Omega}|\Delta w|^{\frac{l+\alpha-1}{\alpha-1}}+C_{3}\int_{\Omega}|\Delta v|^{\frac{l+\alpha-1}{\alpha-1}}+{C_{10}} \nonumber\\
\leq&-\frac{\mu}{16}\int_{\Omega}u^{l+\alpha-1}+{C_{8}}\int_{\Omega}u^{\frac{l+\alpha-1}{\alpha-1}}+C_{3}\int_{\Omega}|\Delta v|^{\frac{l+\alpha-1}{\alpha-1}}+{C_{11}}\nonumber\\
\leq&-\frac{\mu}{32}\int_{\Omega}u^{l+\alpha-1}+C_{3}\int_{\Omega}|\Delta v|^{\frac{l+\alpha-1}{\alpha-1}}+{C_{12}}\quad\mbox{for all}~~t\in(0,T_{max})
\end{align}
with some constants $C_i>0$ $(i={10,11,12})$.

Next, we estimate the second term on the right hand side of the above equation. Since
 $$\alpha>\frac{4N-4+N\sqrt{2N^2-6N+8}}{2N}>\frac{3N-2+\sqrt{N^4-2N^3+N^2+4N+4}}{2N},$$
 then
$\alpha>\frac{4N-4+N\sqrt{2N^2-6N+8}}{2N}$,
we can find $\theta$  fulfilling
$$\frac{N(\alpha-1)}{N\alpha-2N+1}<\theta<\frac{N+2p_0}{p_0+N}.$$
For $\alpha>\frac{3N-2+\sqrt{N^4-2N^3+N^2+4N+4}}{2N}$, we can choose $\theta_1$ satisfying
$$\frac{\alpha-1}{\alpha-2}<\theta_1<1+\frac{2p_0}{N}.$$
Denote $m=\frac{l+\alpha-1}{\alpha-1}$, then there are positive constants $C_i>0$ $(i={13,\cdot\cdot\cdot,19})$ such that
\begin{align}\label{Le3.000038}
&C_{3}\int_{\Omega}|\Delta v|^{m}\nonumber\\
\leq&{C_{13}}\int_{\Omega}|\nabla v \cdot \nabla w+v\Delta w+ u|^{m}\nonumber\\
\leq&{C_{14}}\int_{\Omega}\left[|\nabla v \cdot\nabla w|^{m}+|v\Delta w|^{m}+ u^{m}\right]\nonumber\\
\leq&{C_{15}}\int_{\Omega}|\nabla v |^{m\theta}+{C_{16}}\int_{\Omega}|\nabla w|^{m\theta'}+{C_{17}}\int_{\Omega}v^{m\theta_1}+{C_{18}}\int_{\Omega}|\Delta w|^{m\theta_1'}+{C_{19}}\int_{\Omega}u^{m}
\end{align}
for all $t\in(0,T_{max})$ and $\frac{1}{\theta}+\frac{1}{\theta'}=1$ as well as $\frac{1}{\theta_1}+\frac{1}{\theta'_1}=1$. In order to estimate ${C_{15}}\int_{\Omega}|\nabla v |^{m\theta}$, combining the Gagliardo-Nirenberg interpolation inequality, Young's inequality and Lemma \ref{Le3.2}, we can obtain that there exist constants $C_i>0$ $(i=20,21,22)$ such that
\begin{align}\label{Le3.0000381}
{C_{15}}\parallel\nabla v\parallel^{m\theta}_{L^{m\theta}(\Omega)}
\leq& {C_{20}}(\parallel\Delta v\parallel^{m\theta\xi_3}_{L^{m}(\Omega)}\parallel v\parallel^{m\theta(1-\xi_3)}_{L^{p_0}(\Omega)}+\parallel v\parallel^{m\theta}_{L^{p_0}(\Omega)})\nonumber\\
\leq&{C_{21}}(\parallel\Delta v\parallel^{m\theta\xi_3}_{L^{m}(\Omega)}+1)\nonumber\\
\leq&\frac{C_3}{4}\parallel\Delta v\parallel^{m}_{L^{m}(\Omega)}+{C_{22}}\quad\mbox{for all}~~t\in(0,T_{max}),
\end{align}
where $m\theta\xi_3<m$, $\theta<\frac{N+2p_0}{N+p_{0}}$, $\xi_3=\frac{\frac{1}{N}-\frac{1}{\theta m}+\frac{1}{p_0}}{\frac{2}{N}-\frac{1}{ m}+\frac{1}{p_0}}\in(0,1)$ and $p_0$ is the same as Lemma \ref{Le3.2}.

Similarly, {invoking Young's inequality and the $L^{p}$ theory of elliptic equation, we conclude that there exist
$C_{i}>0$ $(i={23,24,25,26,27})$ such that}
\begin{align}\label{Le3.0000382}
{C_{16}}\parallel\nabla w\parallel^{m\theta'}_{L^{m\theta'}(\Omega)}
\leq& {C_{23}}(\parallel\Delta w\parallel^{m\theta'\xi_4}_{L^{l+\alpha-1}(\Omega)}\parallel w\parallel^{m\theta'(1-\xi_4)}_{L^{l_0}(\Omega)}+\parallel w\parallel^{m\theta'}_{L^{l_0}(\Omega)})\nonumber\\
\leq&{C_{24}}(\parallel\Delta w\parallel^{m\theta'\xi_4}_{L^{l+\alpha-1}(\Omega)}+1)\nonumber\\
\leq&{C_{25}}\parallel\Delta w\parallel^{l+\alpha-1}_{L^{l+\alpha-1}(\Omega)}+{C_{26}}\nonumber\\
\leq&\frac{\mu}{128}\parallel u\parallel^{l+\alpha-1}_{L^{l+\alpha-1}(\Omega)}+{C_{27}}\quad\mbox{for all}~~t\in(0,T_{max}),
\end{align}
where $\theta'<\frac{N(\alpha-1)}{N-1}$, $\xi_4=\frac{\frac{N-1}{N}-\frac{1}{m\theta'}}{1-\frac{1}{l+\alpha-1}}\in(0,1)$.

In order to estimate ${C_{17}}\int_{\Omega}v^{m\theta_1}$ in the estimate \eqref{Le3.000038}. We use Lemma \ref{Le3.2} and Young's inequality, we can derive that
\begin{align}\label{Le3.0000383}
{C_{17}}\parallel v\parallel^{m\theta_1}_{L^{m\theta_1}(\Omega)}
\leq&{C_{28}}(\parallel \Delta v\parallel^{m\theta_1\xi_5}_{L^{m}(\Omega)}\parallel v\parallel^{m\theta_1(1-\xi_5)}_{L^{p_0}(\Omega)}+\parallel v\parallel^{m\theta_1}_{L^{p_0}(\Omega)})\nonumber\\
\leq&{C_{29}}(\parallel \Delta v\parallel^{m\theta_1\xi_5}_{L^{m}(\Omega)}+1)\nonumber\\
\leq&\frac{C_3}{4}\parallel \Delta v\parallel^{m}_{L^{m}(\Omega)}+{C_{30}}\quad\mbox{for all}~~t\in(0,T_{max}),
\end{align}
where $\theta_1<1+\frac{2p_0}{N}$, $\xi_5=\frac{\frac{1}{p_0}-\frac{1}{m\theta_1}}{\frac{2}{N}+\frac{1}{p_0}-\frac{1}{m}}\in(0,1)$ and $C_i$ $(i={28,29,30})$ are positive constants. 

Similarly, {using Young's inequality and the $L^{p}$ theory of elliptic equation}, for any $0<\varepsilon<1$ and $m\theta_1'<l+\alpha-1$, there exist some constants ${C_{31}}>0$ and ${C_{32}}>0$ such that
\begin{align}\label{Le3.0000384}
{C_{18}}\int_{\Omega}|\Delta w|^{m\theta_1'}\leq& \varepsilon\int_{\Omega}|\Delta w|^{l+\alpha-1}+{C_{31}}\nonumber\\
\leq&\frac{\mu}{256}\int_{\Omega}u^{l+\alpha-1}+{C_{32}}\quad\mbox{for all}~~t\in(0,T_{max}),
\end{align}
where $\theta_1'<\frac{l+\alpha-1}{m}$. Moreover,
\begin{align}\label{Le3.0000385}
{C_{19}}\int_{\Omega}u^{m}\leq\frac{\mu}{512}\int_{\Omega}u^{l+\alpha-1} +{C_{33}}\quad\mbox{for all}~~t\in(0,T_{max})
\end{align}
for ${C_{33}}>0$.

Then, inserting \eqref{Le3.0000381}-\eqref{Le3.0000385} into \eqref{Le3.000038}, then there exists a constant ${C_{34}}>0$ such that
\begin{align}\label{Le3.0000386}
C_{3}\int_{\Omega}|\Delta v|^{m}
\leq&\frac{C_3}{2}\parallel\Delta v\parallel^{m}_{L^{m}(\Omega)}+\frac{7\mu}{512}\parallel u\parallel^{l+\alpha-1}_{L^{l+\alpha-1}(\Omega)}+{C_{34}} \quad\mbox{for all}~~t\in(0,T_{max})
\end{align}
for all $t\in(0,T_{max})$. That is, there is constant ${C_{35}}>0$ such that
\begin{align}\label{Le3.00003806}
C_3\int_{\Omega}|\Delta v|^{m}
\leq&\frac{7\mu}{256}\parallel u\parallel^{l+\alpha-1}_{L^{l+\alpha-1}(\Omega)}+{C_{35}} \quad\mbox{for all}~~t\in(0,T_{max}).
\end{align}
Combining \eqref{Le3.00003806} with \eqref{Le3.00038} can draw that there exists a positive constant ${C_{36}}$ such that
\begin{align}\label{Le3.0000387}
&\frac{1}{l}\frac{d}{dt}\int_{\Omega}u^{l} +\frac{l+\alpha-1}{l(\alpha-1)}\int_{\Omega}u^{l} \nonumber\\
\leq& -\frac{\mu}{256}\int_{\Omega}u^{l+\alpha-1}+{C_{36}}\quad\mbox{for all}~~t\in(0,T_{max}).
\end{align}
Then, we have
\begin{equation}
\parallel u\parallel_{L^{l}(\Omega)}\leq C \quad\mbox{for all}~~t\in(0,T_{max}).
\end{equation}
Thus, the proof of this lemma is finished.
\end{proof}

The above lemma shows that the boundedness of $u$ in any $L^{l}(\Omega)$, which implies the boundedness of $\nabla w$ in $L^{\infty}(\Omega)$.
\begin{lemma}\label{Le3.4}
There is constant $C>0$ such that
\begin{equation}
\parallel w(\cdot,t) \parallel_{W^{1,\infty}(\Omega)} \leq C \quad\mbox{for all}~~t\in(0,T_{max}).
\end{equation}
\end{lemma}
\begin{proof}
Combining the $L^{p}$ theory of elliptic equation with Lemma \ref{Le3.03}, for all $t\in(0,T_{max})$, then there exists constant $C_{1}>0$ such that
\begin{equation}\label{Le3.46}
\sup\limits_{t\in(0,T)} \parallel w \parallel_{W^{2,p}(\Omega)} \leq C_{1} \quad\mbox{for any}~~p>N.
\end{equation}
Through the Sobolev embedding theorem, there exists a positive constant $C_{2}$ such that
\begin{equation}\label{Le3.47}
\sup\limits_{t\in(0,T)} \parallel w \parallel_{W^{1,\infty}(\Omega)} \leq C_{2} \quad\mbox{for any}~~t\in(0,T_{max}).
\end{equation}
Then the proof of this lemma is finished.
\end{proof}

In the following, we shall prove the boundedness of $\nabla v$ in $L^{p}$ space to establish the existence of global classical solution.
\begin{lemma}\label{Le3.00000000006}
Under the assumptions in lemma \ref{Le3}, then there exists constant $C>0$ such that
\begin{equation}
\parallel  v(\cdot,t) \parallel_{L^{p}(\Omega)} \leq C \quad\mbox{for all}~~t\in(0,T_{max}).
\end{equation}
\end{lemma}
\begin{proof}
Multiplying the second equation of \eqref{origin} by $v^{p-1}$, using the integrations by
parts as well as the third equation of system \eqref{origin} and Young's inequality, for any $\varepsilon>0$, we have
\begin{align}\label{Le3.00000000007}
&\frac{1}{p}\frac{d}{dt}\int_{\Omega}v^{p}+\int_{\Omega}v^{p}+\left(p-1\right)\int_{\Omega}v^{p-2}|\nabla v|^{2}\nonumber\\
\leq&\frac{p-1}{p}\int_{\Omega}v^{p}w
-\frac{p-1}{p}\int_{\Omega}v^{p}u+\int_{\Omega}v^{p-1}u \nonumber\\
\leq&\varepsilon\int_{\Omega}v^{p+\frac{2}{N}}+C_1\int_{\Omega}w^{\frac{p+\frac{2}{N}}{\frac{2}{N}}}
+C_2 \quad\mbox{for all}~~t\in(0,T_{max}).
\end{align}
In view of \eqref{Lp3}, \eqref{Lp4} and \eqref{Lp42}, there are constants $C_3>0$ and $C_4>0$ such that
\begin{align}\label{Le3.00000000008}
& (p-1)\int_{\Omega}v^{p-2}|\nabla v|^{2}=(p-1)\frac{4}{p^{2}}\parallel |\nabla v|^{\frac{p}{2}}\parallel^{2}_{L^{2}(\Omega)} \quad\mbox{for all}~~t\in(0,T_{max})
\end{align}
and
\begin{align}\label{Le3.00000000000008}
\int_{\Omega}v^{p+\frac{2}{N}}\leq C_{3}\parallel |\nabla v|^{\frac{p}{2}}\parallel^{2}_{L^{2}(\Omega)}+C_{4}\quad\mbox{for all}~~t\in(0,T_{max})
\end{align}
as well as
\begin{align}\label{Le3.00000000009}
\frac{p^{2}}{4}\int_{\Omega}v^{p-2}|\nabla v|^{2}\geq \frac{1}{C_{3}}\int_{\Omega}v^{p+\frac{2}{N}}-\frac{C_4}{C_3}\quad\mbox{for all}~~t\in(0,T_{max}).
\end{align}
Combining \eqref{Le3.00000000007}-\eqref{Le3.00000000009}, Lemma \ref{Le3.4} and the Sobolev embedding theorem, there are positive constants $C_5$ and $C_6$ such that
\begin{align}\label{Le3.000000000010}
&\frac{1}{p}\frac{d}{dt}\int_{\Omega}v^{p}+\int_{\Omega}v^{p}+\frac{\left(p-1\right)}{C_3}\frac{4}{p^{2}}
\int_{\Omega}v^{p+\frac{2}{N}} \nonumber\\
\leq&\varepsilon\int_{\Omega}v^{p+\frac{2}{N}}+C_1\int_{\Omega}w^{\frac{p+\frac{2}{N}}{\frac{2}{N}}}
+C_2\nonumber\\
\leq&\varepsilon\int_{\Omega}v^{p+\frac{2}{N}}+C_5\parallel \nabla w\parallel_{L^{\infty}}
+C_2 \nonumber\\
\leq&\varepsilon\int_{\Omega}v^{p+\frac{2}{N}}+C_6\quad\mbox{for all}~~t\in(0,T_{max}).
\end{align}
Then choosing $\varepsilon <\frac{(p-1)}{C_3}\frac{4}{p^{2}}$, we can derive that
\begin{align}\label{Le3.000000000011}
\frac{1}{p}\frac{d}{dt}\int_{\Omega}v^{p}+\int_{\Omega}v^{p}
\leq C_6 \quad\mbox{for all}~~t\in(0,T_{max}),
\end{align}
which implies that
\begin{equation*}
\int_{\Omega}v^{p}\leq C\quad\mbox{for all}~~t\in(0,T_{max}).
\end{equation*}
Thus, we complete the proof of this lemma.
\end{proof}

Now, the uniform estimate of $u$ has been established, we can use standard semigroup arguments to obtain the uniform bound of $\|\nabla v\|_{L^\infty(\Omega)}$.
\begin{lemma}\label{Le3.6}
Under the assumptions in lemma \ref{Le3}, then there exists a positive constant $C>0$ such that
\begin{equation}
\parallel \nabla v(\cdot,t) \parallel_{L^{\infty}(\Omega)} \leq C \quad\mbox{for all}~~t\in(0,T_{max}).
\end{equation}
\end{lemma}
\begin{proof}
Observing the second equation in system \eqref{origin}, using an associate variation-of-constants formula can draw that
\begin{equation}
v(\cdot,t)=e^{t(\Delta-1)}v_0+\int^{t}_{0}e^{(t-s)(\Delta-1)}[\nabla\cdot(v(\cdot,s)\nabla w(\cdot,s))+u(\cdot,s)]
\end{equation}
for all $t\in(0,T_{max})$. Combining the well-known smoothing property of the Neumann semigroup, Lemma \ref{Le3.03},  Lemma \ref{Le3.4}, Lemma \ref{Le3.00000000006},  and H\"{o}lder inequality, then there exist positive constants $\lambda$ and $C_i$ $(i=0,1,2,3,4)$ such that
\begin{align}\label{Le9}
&\parallel \nabla v \parallel_{L^{\infty}(\Omega)}\nonumber\\
\leq&C_0\parallel \nabla e^{t(\Delta-1)}v_{0} \parallel_{L^{\infty}(\Omega)}\nonumber\\
&+\int_{0}^{t}\parallel\nabla e^{(t-s)(\Delta-1)}\left[\nabla\cdot( v\nabla w)+u   \right]\parallel_{L^{\infty}(\Omega)}\nonumber\\
\leq& C_1\parallel v_{0} \parallel_{W^{1,\infty}(\Omega)}+C_1\int_{0}^{t}\left[1+(t-s)^{-\frac{1}{2}-\frac{N}{2}(\frac{1}{2N}-\frac{1}{\infty})}       \right]e^{-\lambda(t-s)}[\parallel \nabla v\cdot\nabla w+vw\nonumber\\
&-uv+u \parallel_{L^{2N}(\Omega)}+\parallel u \parallel_{L^{2N}(\Omega)}]\nonumber\\
\leq&C_1\parallel v_{0} \parallel_{W^{1,\infty}(\Omega)}+C_1\int_{0}^{t}\left[1+(t-s)^{-\frac{1}{2}-\frac{N}{2}(\frac{1}{2N}-\frac{1}{\infty})}       \right]e^{-\lambda(t-s)}[\parallel\nabla v\parallel_{L^{2N}(\Omega)}\parallel \nabla w \parallel_{L^{\infty}(\Omega)}\nonumber\\
&+\parallel v\parallel_{L^{2N}(\Omega)}\parallel w \parallel_{L^{\infty}(\Omega)}+2\parallel u\parallel_{L^{2N}(\Omega)}]\nonumber\\
\leq& C_2+C_3\int_{0}^{t}\left[1+(t-s)^{-\frac{1}{2}-\frac{N}{2}(\frac{1}{2N}-\frac{1}{\infty})}       \right]e^{-\lambda(t-s)}\parallel\nabla v\parallel_{L^{2N}(\Omega)}+C_4
\end{align}
for all $t\in(0,T_{max})$. Applying the Gagliardo-Nirenberg inequality and Lemma \ref{Le3.1}, we can draw that there exist constants $C_i>0$ $ (i=5,6)$ fulfilling
\begin{align}\label{Le09}
\parallel\nabla v\parallel_{L^{2N}(\Omega)}
&\leq C_5(\parallel\nabla v\parallel^{\zeta}_{L^{\infty}(\Omega)}\parallel v\parallel^{1-\zeta}_{L^{1}(\Omega)}+\parallel v\parallel_{L^{1}(\Omega)})\nonumber\\
&\leq C_6(\parallel\nabla v\parallel^{\zeta}_{L^{\infty}(\Omega)}+1) \quad\mbox{for all}~~t\in(0,T_{max}),
\end{align}
where $\zeta=\frac{1+\frac{1}{2N}}{1+\frac{1}{N}}\in(0,1)$.
Inserting \eqref{Le09} into \eqref{Le9}, we have
\begin{align}\label{Le009}
\parallel \nabla v \parallel_{L^{\infty}(\Omega)}\leq C_7\parallel\nabla v\parallel^{\zeta}_{L^{\infty}(\Omega)}+C_{8}\quad\mbox{for all}~~t\in(0,T_{max})
\end{align}
for some $C_7>0$ and $C_{8}>0$.

 Writing: $$y(t)=\sup\limits_{0\leq t\leq T_{max}}\parallel \nabla v \parallel_{L^{\infty}(\Omega)},$$ then we have
\begin{align}\label{Le0009}
y(t)\leq C_7y^{\zeta}(t)+C_{8} \quad\mbox{for all}~~t\in(0,T_{max}).
\end{align}
Then using a basic calculation, the conclusion of this lemma is valid.
\end{proof}

Similarly, we can obtain the uniform estimate of $\|u\|_{L^\infty(\Omega)}$ for the case of $N\geq4$.
\begin{lemma}\label{Le3.7}
Let $\alpha>\frac{N}{2}$. If $N\geq4$, then there exists a positive constant $C>0$ such that
\begin{equation}
\parallel u(\cdot,t) \parallel_{L^{\infty}(\Omega)} \leq C \quad\mbox{for all}~~t\in(0,T_{max}).
\end{equation}
\end{lemma}
\begin{proof}
Recalling the first equation in \eqref{origin}, we have
$$u_t=\Delta u-\nabla\cdot(u\nabla v)+\xi\nabla\cdot(u\nabla w)+au-\mu u^{\alpha} \quad\mbox{for all}~~t\in(0,T_{max}).$$
Using the variation-of-constants formula can draw that
\begin{align}\label{Le3.7602}
u(\cdot,t)=e^{(t-t_{0})\Delta}u(\cdot,t_{0})-\int^{t}_{t_{0}}e^{(t-s)\Delta}\nabla h(\cdot,s)+\int^{t}_{t_{0}}e^{(t-s)\Delta}
f(u(\cdot,s))
\end{align}
for all $t\in(t_0,T_{max})$, where  $t_{0}: =(t-1)_{+}$, $h:=-u\nabla v+\xi u\nabla w$ and $f(u):=au-\mu u^{\alpha}.$

On the other hand, employing Lemma \ref{Le3.03}, Lemma \ref{Le3.4} and Lemma \ref{Le3.6}, we can arrive at the conclusion that
\begin{align}\label{Le3.763}
&\parallel h \parallel_{L^{2N}(\Omega)} \nonumber\\
=& \parallel  -u\nabla v+\xi u\nabla w \parallel_{L^{2N}(\Omega)} \nonumber\\
\leq & \parallel  u\parallel_{L^{2N}(\Omega)}\parallel \nabla v \parallel_{L^{\infty}(\Omega)}+|\xi|\parallel u \parallel_{L^{2N}(\Omega)}\parallel \nabla w \parallel_{L^{\infty}(\Omega)}\nonumber\\
\leq & C_1 \quad\mbox{for all}~~t\in(0,T_{max})
\end{align}
with a positive constant $C_1$.
Hence
\begin{align}\label{Le3.7005}
\parallel u\parallel_{L^{\infty}(\Omega)}
\leq& \parallel e^{(t-t_{0})\Delta}u(\cdot,t_{0})\parallel_{L^{\infty}(\Omega)}\nonumber\\
&+\int^{t}_{t_{0}}\parallel e^{(t-s)\Delta}\nabla h(\cdot,s)\parallel_{L^{\infty}(\Omega)} \nonumber\\ &+\int^{t}_{t_{0}}\parallel e^{(t-s)\Delta}f(u(\cdot,s))\parallel_{L^{\infty}(\Omega)} \quad\mbox{for all}~~t\in(0,T_{max}).
\end{align}
If $t\in(0,1]$, based on the maximum principle, we can deduce that
\begin{equation}
\parallel e^{(t-t_{0})\Delta}u_{0}\parallel_{L^{\infty}(\Omega)}\leq\parallel u_0\parallel_{L^{\infty}(\Omega)}\quad\mbox{for all}~~t\in(t_0,T_{max}).
\end{equation}
If $t>1$, by using an application of the Neumann heat semigroup $(e^{t\Delta})_{t>0}$ on $\Omega$ (\cite{FKI 1991 DCDS,winkler 2010 JDE}) and Lemma \ref{Le3.1}, then there exist constants $C_2>0$ and $C_3>0$, we can obtain that
\begin{equation}
\parallel e^{(t-t_{0})\Delta}u_{0}\parallel_{L^{\infty}(\Omega)}\leq C_2(t-t_0)^{-\frac{N}{2}}\parallel u_0\parallel_{L^{1}(\Omega)}
\leq C_3\quad\mbox{for all}~~t\in(t_0,T_{max}).
\end{equation}
Subsequently, in view of \eqref{Le3.763}, using the smoothing properties of the Stokes semigroup can find constants $\lambda_1>0$ and $C_4>0$ as well as
$C_5>0$ such that
\begin{align}
&\int^{t}_{t_{0}}\parallel e^{(t-s)\Delta}\nabla h(\cdot,s)\parallel_{L^{\infty}(\Omega)}\nonumber\\
\leq& C_4\int^{t}_{t_{0}}\left[1+(t-s)^{-\frac{1}{2}-\frac{N}{2}\times\frac{1}{2N}}\right]e^{-\lambda_1(t-s)}
\parallel h(\cdot,s)\parallel_{L^{2N}(\Omega)} \nonumber\\
\leq& C_5 \quad\mbox{for all}~~t\in(t_0,T_{max}).
\end{align}
Similarly, we can estimate the third term on the right hand side of \eqref{Le3.7005} as follows:
\begin{align}\label{Le3.7007}
&\int^{t}_{t_{0}}\parallel e^{(t-s)\Delta}f(u(\cdot,s))\parallel_{L^{\infty}(\Omega)}\\
\leq&\int^{t}_{t_{0}}\sup_{s\geq0}\left(as-\mu s^{\alpha}\right)_{+}
\leq\int^{t}_{t_{0}}\frac{a^{\frac{\alpha}{\alpha-1}}_{+}}{\mu^{\frac{1}{\alpha-1}}}\alpha^{\frac{-1}{\alpha-1}}(1-\frac{1}{\alpha})
\leq\frac{a^{\frac{\alpha}{\alpha-1}}_{+}}{\mu^{\frac{1}{\alpha-1}}}\alpha^{\frac{-1}{\alpha-1}}(1-\frac{1}{\alpha})\nonumber
\quad\mbox{for all}~~t\in(t_0,T_{max})
\end{align}
by using $0\leq t-t_0\leq1$. Collecting \eqref{Le3.7005}-\eqref{Le3.7007}, there exists a constant $C_6>0$ fulfilling
\begin{equation}
\parallel u \parallel_{L^{\infty}(\Omega)} \leq C_6 \quad\mbox{for all}~~t\in(0,T_{max}).
\end{equation}
So the proof of this lemma is completed.
\end{proof}

\newcommand\HIIII{{\bf IIII}}
\section{A priori estimates for $N=3$ and $N=2$}
\setcounter{equation}{0}
In this section, we give some estimates about the local solution $(u,v,w)$ which was obtained in Lemma \ref{Le3} for the case of $N=3$ and $N=2$.

\begin{lemma}\label{Le3.3010}
If $N\leq3$, then we can find a positive constant $C$ such that
\begin{equation}
\parallel v(\cdot,t) \parallel_{L^\infty(\Omega)} \leq C \quad\mbox{for all}~~t\in(0,T_{max}).
\end{equation}
\end{lemma}
\begin{proof}
Let $T_{max}(u_0,v_0)$ and $(u,v,w)$  are the same as in Lemma \ref{Le3}, for integers $j\geq0$, denote $p_{j}:=2^{j}$ and
\begin{align}\label{203}
M_{j}(T):=1+\max_{t\in[0,T]}\int_{\Omega}v^{p_{j}}(\cdot,t),~~T\in(0,T_{max}(u_0,v_0)).
\end{align}
For convenience, let $p=p_j$ with $j\geq1$. Multiplying both sides of the second equation in \eqref{origin} by $v^{p-1}$, and integrating by parts, we have
\begin{align}
\int_{\Omega}v^{p-1}v_{t}=\int_{\Omega}v^{p-1}\Delta v+\int_{\Omega}v^{p-1}\nabla\cdot(v\nabla w)-\int_{\Omega}v^{p}+\int_{\Omega}v^{p-1}u
\end{align}
for all $t\in(0,T_{max})$. Using the identity $\Delta w=w-u$ to infer that
\begin{align}\label{200}
&\frac{1}{p}\frac{d}{dt}\int_{\Omega}v^{p}+(p-1)\int_{\Omega}v^{p-2}|\nabla v|^{2}\nonumber\\
=&\frac{p-1}{p}\int_{\Omega}v^{p}\Delta w+\int_{\Omega}v^{p-1}u-\int_{\Omega}v^{p}\nonumber\\
=&\frac{p-1}{p}\int_{\Omega}v^{p}(w-u)+\int_{\Omega}v^{p-1}u-\int_{\Omega}v^{p}\nonumber\\
=&\frac{p-1}{p}\int_{\Omega}v^{p}w-\frac{p-1}{p}\int_{\Omega}v^{p}u+\int_{\Omega}v^{p-1}u-\int_{\Omega}v^{p} \quad\mbox{for all}~~t\in(0,T_{max}),
\end{align}
where $p\geq2$, one can easily obtain that
\begin{align}\label{199}
(p-1)\int_{\Omega}v^{p-2}|\nabla v|^{2}\geq \frac{p}{2}\int_{\Omega}v^{p-2}|\nabla v|^{2}=\frac{2}{p}\int_{\Omega}|\nabla v^{\frac{p}{2}}|^{2}\quad\mbox{for all}~~t\in(0,T_{max}).
\end{align}
Using the Young's inequality and Lemma \ref{Le3.1}, we have
\begin{align}\label{201}
\int_{\Omega}uv^{p-1}\leq\frac{p-1}{p}\int_{\Omega}uv^{p}+\frac{1}{p}\int_{\Omega}u\leq\frac{p-1}{p}\int_{\Omega}uv^{p}+C_1
\quad\mbox{for all}~~t\in(0,T_{max})
\end{align}
with certain positive $C_1>0$. For the first term on the right hand side of \eqref{200}, by using the Cauchy-Schwarz inequality, the Gagliardo-Nirenberg inequality and Lemma \ref{Le030.2}, we can draw that
\begin{align}\label{202}
\frac{p-1}{p}\int_{\Omega}v^{p}w
\leq& \int_{\Omega}v^{p}w
\leq \left\{\int_{\Omega}v^{2p}\right\}^{\frac{1}{2}}\cdot\left\{ \int_{\Omega}w^{2} \right\}^{\frac{1}{2}}\nonumber\\
\leq&C_2\left\{ \int_{\Omega}v^{2p} \right\}^{\frac{1}{2}}
=C_2\parallel v^{\frac{p}{2}} \parallel^{2}_{L^{4}(\Omega)}\nonumber\\
\leq&C_3\parallel \nabla v^{\frac{p}{2}} \parallel^{2\theta}_{L^{2}(\Omega)}\parallel v^{\frac{p}{2}} \parallel^{2(1-\theta)}_{L^1(\Omega)}+C_3\parallel v^{\frac{p}{2}}\parallel^{2}_{L^1(\Omega)}\quad\mbox{for all}~~t\in(0,T_{max}),
\end{align}
where $\theta=\frac{3N}{2N+4}\in(0,1)$. As herein for any $T\in(0,T_{max})$, we have
\begin{align}\label{204}
\parallel v^{\frac{p}{2}} \parallel_{L^1(\Omega)}=\int_{\Omega}v^{p_{j-1}}\leq M_{j-1}(T) \quad\mbox{for all}~~t\in(0,T_{max}).
\end{align}
Hence, using the Young's inequality, it holds that
\begin{align}\label{205}
\frac{p-1}{p}\int_{\Omega}v^{p}w
\leq & C_3M^{2(1-\theta)}_{j-1}(T)\parallel \nabla v^{\frac{p}{2}} \parallel^{2\theta}_{L^{2}(\Omega)}+C_3M^{2}_{j-1}(T)\nonumber\\
=&\{ \frac{2}{p}\int_{\Omega}|\nabla v^{\frac{p}{2}}|^{2}\}^{\theta}\cdot\big\{C_3\cdot(\frac{p}{2})^{\theta}\cdot M^{2(1-\theta)}_{j-1}(T)\big\}+C_3M^{2}_{j-1}(T)\nonumber\\
\leq&\frac{2}{p}\int_{\Omega}|\nabla v^{\frac{p}{2}}|^{2}
+2^{-\frac{\theta}{1-\theta}}C_3^{\frac{1}{1-\theta}}p^{\frac{\theta}{1-\theta}}M^{2}_{j-1}(T)+C_3M^{2}_{j-1}(T)\nonumber\\
\leq&\frac{2}{p}\int_{\Omega}|\nabla v^{\frac{p}{2}}|^{2}+C_4p^{\frac{\theta}{1-\theta}}M^{2}_{j-1}(T)
\quad\mbox{for all}~~t\in(0,T_{max})
\end{align}
with $C_4=2^{-\frac{\theta}{1-\theta}}C_3^{\frac{1}{1-\theta}}+C_3$ and $p\geq2$. Then, combining \eqref{200}, \eqref{199}, \eqref{201} with \eqref{205} can derive that
\begin{align}\label{206}
\frac{1}{p}\frac{d}{dt}\int_{\Omega}v^{p}+\int_{\Omega}v^{p}
\leq&C_1+C_4p^{\frac{\theta}{1-\theta}}M^{2}_{j-1}(T)\nonumber\\
\leq&C_5p^{\frac{\theta}{1-\theta}}M^{2}_{j-1}(T)\quad\mbox{for all}~~t\in(0,T_{max}).
\end{align}
Therefore, applying an ODE comparison argument entails that
\begin{align}\label{207}
\int_{\Omega}v^{p}\leq \max\left\{ \int_{\Omega}v_0^{p}, C_5p^{\frac{1}{1-\theta}}M^{2}_{j-1}(T)\right\}\quad\mbox{for all}~~t\in(0,T_{max}).
\end{align}
For any $j\geq1$, we have
\begin{align}\label{208}
M_j(T)\leq1+\max\left\{\int_{\Omega}v_0^{p_j}, C_5p_j^{\frac{1}{1-\theta}}M^{2}_{j-1}(T)\right\}\quad\mbox{for all}~~t\in(0,T_{max}),
\end{align}
which we use the following estimate
\begin{align}\label{209}
\frac{d}{dt}\int_{\Omega}v+\int_{\Omega}v=\int_{\Omega}u \quad\mbox{for all}~~t\in(0,T_{max}(u_0,v_0)).
\end{align}
 By an ODE comparison again, $\int_{\Omega}v\leq C_6$, we easily get $M_0(T)\leq1+C_6$. Let
$$  \overline{M_j}=\sup\limits_{T\in(0,T_{max}(u_0,v_0))}M_j(T)  $$
represents a sequence $(\overline{M_j})_{j\geq0}$ of finite numbers $\overline{M_j}\geq1$ which satisfies
\begin{align}\label{210}
\overline{M_0}\leq1+C_6  ~~~\mbox{and}~~~~ \overline{M_j}\leq1+max\left\{\int_{\Omega}v_0^{p_j},C_5p_j^{\frac{1}{1-\theta}}M^{2}_{j-1}(T)\right\}\quad\mbox{for all}~~j\geq1.
\end{align}
Now the remainder is quite standard: if $\overline{M_j}\leq1+\int_{\Omega}v_0^{p_j}$ for infinitely many $j\geq0$, then from \eqref{203}, it directly follows that
\begin{align}\label{2110000}
\parallel v\parallel_{L^\infty(\Omega)}\leq\liminf_{j\rightarrow\infty}(\overline{M_j}-1)^{\frac{1}{p_j}}
\leq\liminf_{j\rightarrow\infty}\{\int_{\Omega}v_0^{p_j}\}^{\frac{1}{p_j}}=\parallel v_0\parallel_{L^\infty(\Omega)}
\end{align}
for all $t\in(0,T_{max}(u_0,v_0))$. Otherwise, \eqref{210} warrants the existence of $j_0 \in N$ such that
\begin{align}\label{212}
\overline{M_j}\leq C_5p_j^{\frac{1}{1-\theta}}\overline{M}^{2}_{j-1}(T)\quad\mbox{for all}~~j\geq j_0.
\end{align}
We can find a positive constant $C_7 > 1$, which independents of $\xi$ by construction of $(\overline{M_j})_{j\geq0}$ and satisfies
\begin{align*}
\overline{M_j}\leq C_7^{j}\overline{M}^{2}_{j-1}\quad\mbox{for all}~~j\geq 1.
\end{align*}
Hence, by a straightforward induction, we can infer that
\begin{align*}
\overline{M_j}\leq C_7^{2^{j+1}-j-2}\overline{M}^{2^j}_0\leq C_7^{2^{j+1}}\overline{M}^{2^j}_0\quad\mbox{for all}~~j\geq 1.
\end{align*}
According to the estimate of \eqref{210}, we have
\begin{equation*}
\parallel v(\cdot,t)\parallel_{L^\infty(\Omega)}\leq\liminf_{j\rightarrow\infty}(\overline{M_j})^{\frac{1}{2_j}}_j
\leq C_7^{2}\overline{M}_0\leq C_7^{2}(1+C_6) \quad\mbox{for all}~~t\in(0,T_{max}(u_0,v_0)).
\end{equation*}
Thus, we accomplished the proof of the Lemma.
\end{proof}

As a basic step of a priori estimates for the case $N=2$ and $N=3$, we establish the main inequality by applying standard testing procedures to the first equation in system \eqref{origin}.
\begin{lemma}\label{Le3.301121}
Under the conditions of Lemma \ref{Le3}. Let $N=3$, $\alpha>2$, for any $\mu>0$ or $\alpha=2$, $\mu>\max\{2|\xi|,2C_3A\}$, where $C_3$ and $A$ can be found in next lemma. Without any restrictions on the index $\xi$, one can find a constant $C>0$ such that
\begin{equation}
\frac{d}{dt}\int_{\Omega}u^{2}+\int_{\Omega}|\nabla u|^{2}+\frac{\mu}{2}\int_{\Omega}u^{\alpha+1}\leq C\int_{\Omega}|\nabla v|^{6}+C
 \quad\mbox{for all}~~t\in(0,T_{max}).
\end{equation}
\end{lemma}
\begin{proof} We divide the proof of this lemma into two cases. \\
\textbf{Case \uppercase\expandafter{\romannumeral 1}: $\alpha>2$.} Multiplying the first equation in \eqref{origin} by $u$ and integrating by parts, then using that $\Delta w=w-u$ for any $\xi\in \mathbb{R}$, for all $t\in(0,T_{max})$, we can get
\begin{align}\label{40}
&\frac{1}{2}\frac{d}{dt}\int_{\Omega}u^{2}+\int_{\Omega}|\nabla u|^{2}\nonumber\\
=&\int_{\Omega}u\nabla u\cdot\nabla v-\xi\int_{\Omega}u\nabla u\cdot\nabla w+a\int_{\Omega}u^{2}-\mu\int_{\Omega}u^{\alpha+1}\nonumber\\
=&\int_{\Omega}u\nabla u\cdot\nabla v-\frac{\xi}{2}\int_{\Omega}\nabla u^{2}\cdot\nabla w+a\int_{\Omega}u^{2}-\mu\int_{\Omega}u^{\alpha+1}\nonumber\\
=&\int_{\Omega}u\nabla u\cdot\nabla v+\frac{\xi}{2}\int_{\Omega} u^{2}\Delta w+a\int_{\Omega}u^{2}-\mu\int_{\Omega}u^{\alpha+1}\nonumber\\
=&\int_{\Omega}u\nabla u\cdot\nabla v+\frac{\xi}{2}\int_{\Omega} u^{2}(w-u)+a\int_{\Omega}u^{2}-\mu\int_{\Omega}u^{\alpha+1}\nonumber\\
\leq&\int_{\Omega}u\nabla u\cdot\nabla v+\frac{|\xi|}{2}\int_{\Omega} u^{2}w+\frac{|\xi|}{2}\int_{\Omega} u^{3}+a\int_{\Omega}u^{2}-\mu\int_{\Omega}u^{\alpha+1}.
\end{align}
In coping with the first term on the right hand side of \eqref{40}, by the Young's inequality, we can find constants $C_i>0$ $(i=1,2,3)$ such that
\begin{align}\label{41}
\int_{\Omega}u\nabla u\cdot\nabla v\leq&\frac{1}{2}\int_{\Omega}|\nabla u|^{2}+\frac{1}{2}\int_{\Omega}u^{2}|\nabla v|^{2}\nonumber\\
\leq&\frac{1}{2}\int_{\Omega}|\nabla u|^{2}+C_1\int_{\Omega}u^{3}+C_2\int_{\Omega}|\nabla v|^{6}\nonumber\\
\leq&\frac{1}{2}\int_{\Omega}|\nabla u|^{2}+\frac{\mu}{2}\int_{\Omega}u^{\alpha+1}+C_2\int_{\Omega}|\nabla v|^{6}+C_3\quad\mbox{for all}~~t\in(0,T_{max}).
\end{align}
Using the Young's inequality again, there exist some constants $C_i>0$ $(i=4,5,6)$ such that
\begin{align}\label{42}
\frac{|\xi|}{2}\int_{\Omega} u^{2}w
\leq&C_4\int_{\Omega} u^{3}+C_5\int_{\Omega}w^{3}\nonumber\\
\leq&\frac{\mu}{16}\int_{\Omega} u^{\alpha+1}+C_5\int_{\Omega}w^{3}+C_6\quad\mbox{for all}~~t\in(0,T_{max}).
\end{align}
Recalling Lemma \ref{Le030.2}, there exists a positive $C_7>0$ such that
\begin{align}\label{43}
\parallel w\parallel_{L^{l_0}(\Omega)}\leq C_7\quad\mbox{for all}~~t\in(0,T_{max}).
\end{align}
Invoking the Gagliardo-Nirenberg inequality along with standard elliptic regularity theory (\cite{DN2001JMAA}) can pick constant $C_8>0$ such that
\begin{align}\label{44}
C_5\int_{\Omega}w^{3}\leq& C_5C_8C_7^{3(1-\theta)}\parallel-\Delta w+w\parallel^{3\theta}_{L^{3}(\Omega)}\nonumber\\
=& C_5C_8C_7^{3(1-\theta)}\parallel u\parallel^{3\theta}_{L^{3}(\Omega)}\quad\mbox{for all}~~t\in(0,T_{max})
\end{align}
with $\theta=\frac{N}{N+12}\in(0,1)$. Applying the Young's inequality, owing to $\alpha>2$ and $0<\theta<1$, there are constants $C_9>0$ and $C_{10}>0$ as well as $C_{11}>0$ such that
\begin{align}\label{45}
C_5\int_{\Omega}w^{3}\leq&C_9\parallel u\parallel^{3}_{L^{3}(\Omega)}+C_{10}
\leq\frac{\mu}{16}\int_{\Omega}u^{\alpha+1}+C_{11}\quad\mbox{for all}~~t\in(0,T_{max}).
\end{align}
Therefore, combining \eqref{42} with \eqref{45}, in view of $\alpha>2$, we can find a constant $C_{12}>0$ such that
\begin{align}\label{46}
\frac{|\xi|}{2}\int_{\Omega} u^{2}w\leq\frac{\mu}{8}\int_{\Omega} u^{\alpha+1}+C_{12}\quad\mbox{for all}~~t\in(0,T_{max}).
\end{align}
 Using the Young's inequality and $\alpha>2$, there is a constant $C_{13}>0$ such that
 \begin{align}\label{47}
a\int_{\Omega}u^{2}\leq\frac{\mu}{16}\int_{\Omega}u^{\alpha+1}+C_{13}\quad\mbox{for all}~~t\in(0,T_{max}).
\end{align}
 By the Young's inequality and $\alpha>2$ again, we obtain
\begin{align}\label{49}
\frac{|\xi|}{2}\int_{\Omega} u^{3}\leq \frac{\mu}{16}\int_{\Omega}u^{\alpha+1}+C_{14}\quad\mbox{for all}~~t\in(0,T_{max})
\end{align}
with $C_{14}>0$.
Putting \eqref{41},  \eqref{46} and \eqref{47} as well as \eqref{49} into \eqref{40}, then there exists a positive constant $C_{15}>0$ such that
\begin{align*}
\frac{d}{dt}\int_{\Omega}u^{2}+\int_{\Omega}|\nabla u|^{2}+\frac{\mu}{2}\int_{\Omega}u^{\alpha+1}
\leq2C_2\int_{\Omega}|\nabla v|^{6}+C_{15}\quad\mbox{for all}~~t\in(0,T_{max}).
\end{align*}
\textbf{Case \uppercase\expandafter{\romannumeral 2}: $\alpha=2$.} Recalling \eqref{40}, we have
\begin{align}\label{9940}
&\frac{1}{2}\frac{d}{dt}\int_{\Omega}u^{2}+\int_{\Omega}|\nabla u|^{2}\nonumber\\
=&\int_{\Omega}u\nabla u\cdot\nabla v-\xi\int_{\Omega}u\nabla u\cdot\nabla w+a\int_{\Omega}u^{2}-\mu\int_{\Omega}u^{3}\nonumber\\
=&\int_{\Omega}u\nabla u\cdot\nabla v-\frac{\xi}{2}\int_{\Omega}\nabla u^{2}\cdot\nabla w+a\int_{\Omega}u^{2}-\mu\int_{\Omega}u^{3}\nonumber\\
=&\int_{\Omega}u\nabla u\cdot\nabla v+\frac{\xi}{2}\int_{\Omega} u^{2}\Delta w+a\int_{\Omega}u^{2}-\mu\int_{\Omega}u^{3}\nonumber\\
=&\int_{\Omega}u\nabla u\cdot\nabla v+\frac{\xi}{2}\int_{\Omega} u^{2}(w-u)+a\int_{\Omega}u^{2}-\mu\int_{\Omega}u^{3}\nonumber\\
\leq&\int_{\Omega}u\nabla u\cdot\nabla v+\frac{|\xi|}{2}\int_{\Omega} u^{2}w+\frac{|\xi|}{2}\int_{\Omega} u^{3}+a\int_{\Omega}u^{2}-\mu\int_{\Omega}u^{3} \quad\mbox{for all}~~t\in(0,T_{max}).
\end{align}
For the first term on the right hand side of \eqref{9940}, we have
\begin{align}\label{9941}
\int_{\Omega}u\nabla u\cdot\nabla v\leq&\frac{1}{2}\int_{\Omega}|\nabla u|^{2}+\frac{1}{2}\int_{\Omega}u^{2}|\nabla v|^{2}\nonumber\\
\leq&\frac{1}{2}\int_{\Omega}|\nabla u|^{2}+\frac{\mu}{16}\int_{\Omega}u^{3}+C_1\int_{\Omega}|\nabla v|^{6}\quad\mbox{for all}~~t\in(0,T_{max})
\end{align}
with $C_1>0$.
Using the Young's inequality can derive that there is constant $C_2>0$ such that
\begin{align}\label{9947}
a\int_{\Omega}u^{2}\leq\frac{\mu}{16}\int_{\Omega}u^{3}+C_{2}\quad\mbox{for all}~~t\in(0,T_{max}).
\end{align}
In view of \eqref{42}-\eqref{46}, there exists a positive constant $C_3>0$ such that
\begin{align}\label{99474899}
\frac{|\xi|}{2}\int_{\Omega} u^{2}w\leq\frac{5\mu}{8}\int_{\Omega} u^{3}+C_{3}\quad\mbox{for all}~~t\in(0,T_{max}).
\end{align}
Combining the above inequalities, there is a positive constant $C_4>0$, which implies that
\begin{align}\label{99409999999}
&\frac{1}{2}\frac{d}{dt}\int_{\Omega}u^{2}+\int_{\Omega}|\nabla u|^{2}\nonumber\\
\leq&\int_{\Omega}u\nabla u\nabla v+\frac{|\xi|}{2}\int_{\Omega} u^{2}w+\frac{|\xi|}{2}\int_{\Omega} u^{3}+a\int_{\Omega}u^{2}-\mu\int_{\Omega}u^{3}\nonumber\\
\leq&\frac{1}{2}\int_{\Omega}|\nabla u|^{2}+\frac{\mu}{16}\int_{\Omega}u^{3}+C_1\int_{\Omega}|\nabla v|^{6}+\frac{\mu}{16}\int_{\Omega}u^{3}+\frac{5\mu}{8}\int_{\Omega} u^{3}+\frac{|\xi|}{2}\int_{\Omega}u^{3}-\mu\int_{\Omega}u^{3}+C_{4}\nonumber\\
=&\frac{1}{2}\int_{\Omega}|\nabla u|^{2}+C_1\int_{\Omega}|\nabla v|^{6}+\frac{|\xi|}{2}\int_{\Omega}u^{3}-\frac{\mu}{4}\int_{\Omega}u^{3}+C_{4} \quad\mbox{for all}~~t\in(0,T_{max}).
\end{align}
Let $\mu>2|\xi|$, we have
\begin{align}\label{994099999991}
\frac{d}{dt}\int_{\Omega}u^{2}+\int_{\Omega}|\nabla u|^{2}+\frac{\mu}{2}\int_{\Omega}u^{3}
\leq2C_1\int_{\Omega}|\nabla v|^{6}+C_{5}\quad\mbox{for all}~~t\in(0,T_{max})
\end{align}
with some positive constant $C_5>0$. Then the proof of Lemma \ref{Le3.301121} is completed.
\end{proof}

By using the standard elliptic regularity theory, we have the following estimates.
\begin{lemma}\label{Le3.930112}
Let $N=3$, there is a constant $C>0$ such that
\begin{equation}
\frac{d}{dt}\int_{\Omega}|\nabla v|^{4}+2\int_{\Omega}|\nabla v|^{2}|D^{2} v|^{2}+4\int_{\Omega}|\nabla v|^{4} \leq C \int_{\Omega}u^{3} \quad\mbox{for all}~~t\in(0,T_{max}).
\end{equation}
\end{lemma}
\begin{proof}
Recalling Lemma \ref{Le3.3010}, there is a constant $C_1>0$ such as
\begin{equation}\label{58}
\parallel v\parallel_{L^{\infty}(\Omega)}\leq C_1\quad\mbox{for all}~~t\in(0,T_{max}).
\end{equation}
In view of the standard elliptic regularity theory, we can find a constant $C_2>0$ satisfies
\begin{align}\label{59}
\parallel D^{2}v \parallel_{L^{3}(\Omega)}\leq C_2\parallel -\Delta v+v \parallel_{L^{3}(\Omega)}\quad\mbox{for all}~~v\in W^{2,3}(\Omega) \quad\mbox{such that}~~ \frac{\partial v}{\partial\nu}|_{\partial\Omega}=0.
\end{align}
Integrating by parts in the second equation from \eqref{origin}, we obtain that
\begin{align}\label{60}
\frac{1}{4}\frac{d}{dt}\int_{\Omega}|\nabla v|^{4}=&\int_{\Omega}|\nabla v|^{2}\nabla v\cdot\nabla\{\Delta v+\nabla\cdot(v\nabla w)-v+u\}\nonumber\\
=&\frac{1}{2}\int_{\Omega}|\nabla v|^{2}\Delta|\nabla v|^{2}-\int_{\Omega}|\nabla v|^{2}|D^{2}v|^{2}+\int_{\Omega}|\nabla v|^{2}\nabla v\cdot\nabla\cdot(\nabla v\cdot\nabla w)\nonumber\\
&+\int_{\Omega}|\nabla v|^{2}\nabla v\cdot\nabla\cdot(v\Delta w)-\int_{\Omega}|\nabla v|^{4}+\int_{\Omega}|\nabla v|^{2}\nabla v\cdot\nabla u\nonumber\\
=&-\frac{1}{2}\int_{\Omega}|\nabla|\nabla v|^{2}|^{2}+\frac{1}{2}\int_{\partial\Omega}|\nabla v|^{2}\frac{\partial |\nabla v|^{2}}{\partial\nu}-\int_{\Omega}|\nabla v|^{2}|D^{2}v|^{2}\nonumber\\
&+\int_{\Omega}|\nabla v|^{2}\nabla v\cdot\nabla\cdot(\nabla v\cdot\nabla w)-\int_{\Omega}v\Delta w\nabla\cdot(|\nabla v|^{2}\nabla v)\nonumber\\
&-\int_{\Omega}|\nabla v|^{4}-\int_{\Omega}u\cdot\{2\nabla v\cdot(|D^{2}v|^{2}\cdot\nabla v)+|\nabla v|^{2}\Delta v\}\nonumber\\
\leq&-\int_{\Omega}|\nabla v|^{2}|D^{2}v|^{2}+\int_{\Omega}|\nabla v|^{2}\nabla v\cdot\nabla\cdot(\nabla v\cdot\nabla w)-\int_{\Omega}v\Delta w\nabla\cdot(|\nabla v|^{2}\nabla v)\nonumber\\
&-\int_{\Omega}|\nabla v|^{4}+(2+\sqrt{N})\int_{\Omega}u|\nabla v|^{2}|D^{2}v|\quad\mbox{for all}~~t\in(0,T_{max}),
\end{align}
where we use the fact that
\begin{equation}\label{61}
|\Delta v|\leq\sqrt{N}|D^{2}v|~~\mbox{in} ~~\Omega~~\quad\mbox{for all}~~v\in C^{2}(\Omega).
\end{equation}
Combining the method which was given by Tao and Winkler \cite{Tao Winkler 2021 NA}, as well as the estimates \eqref{58} and \eqref{59}, we can draw that
\begin{align}\label{62}
&\int_{\Omega}|\nabla v|^{2}\nabla v\cdot\nabla\cdot(\nabla v\cdot\nabla w)\nonumber\\
\leq&\big(1+\frac{\sqrt{N}}{4}\big)\parallel\nabla v\parallel^{4}_{L^{6}(\Omega)}\parallel D^{2}w\parallel_{L^{3}(\Omega)}\nonumber\\
\leq&\big(1+\frac{\sqrt{N}}{4}\big)(4+\sqrt{N})^{\frac{4}{3}}\parallel v\parallel^{\frac{4}{3}}_{L^{\infty}(\Omega)}
\cdot\left\{\int_{\Omega}|\nabla v|^{2}|D^{2}v|^{2}\right\}^{\frac{2}{3}}\cdot\parallel D^{2}w\parallel_{L^{3}(\Omega)}\nonumber\\
\leq&\big(1+\frac{\sqrt{N}}{4}\big)(4+\sqrt{N})^{\frac{4}{3}}C_1^{\frac{4}{3}}C_2\cdot\left\{\int_{\Omega}|\nabla v|^{2}|D^{2}v|^{2}\right\}^{\frac{2}{3}}\cdot\parallel u\parallel_{L^{3}(\Omega)}
\end{align}
for all $t\in(0,T_{max})$. Similarly, we also can get that
\begin{align}\label{63}
&-\int_{\Omega}v\Delta w\nabla\cdot(|\nabla v|^{2}\nabla v)\nonumber\\
\leq&(2+\sqrt{N})\parallel v \parallel_{L^{\infty}(\Omega)}\parallel \nabla v\parallel_{L^{6}(\Omega)}\cdot\left\{\int_{\Omega}|\nabla v|^{2}|D^{2}v|^{2}\right\}^{\frac{1}{2}}\cdot\parallel\Delta w\parallel_{L^{3}(\Omega)}\nonumber\\
\leq&(2+\sqrt{N})(4+\sqrt{N})^{\frac{1}{3}}\parallel v\parallel^{\frac{4}{3}}_{L^{\infty}(\Omega)}\cdot\left\{\int_{\Omega}|\nabla v|^{2}|D^{2}v|^{2}\right\}^{\frac{2}{3}}\cdot\parallel\Delta w\parallel_{L^{3}(\Omega)}\nonumber\\
\leq&(2+\sqrt{N})(4+\sqrt{N})^{\frac{1}{3}}\sqrt{N}C_1^{\frac{4}{3}}C_2\cdot\left\{\int_{\Omega}|\nabla v|^{2}|D^{2}v|^{2}\right\}^{\frac{2}{3}}\cdot\parallel u\parallel_{L^{3}(\Omega)}
\end{align}
for all $t\in(0,T_{max})$. In view of the H\"{o}lder inequality, we have
\begin{align}\label{64}
&(2+\sqrt{N})\int_{\Omega}u|\nabla v|^{2}|D^{2}v|\nonumber\\
\leq&(2+\sqrt{N})\cdot\left\{\int_{\Omega} u^{2}|\nabla v|^{2}\right\}^{\frac{1}{2}}\cdot\left\{\int_{\Omega}|\nabla v|^{2}|D^{2}v|^{2}\right\}^{\frac{1}{2}}\nonumber\\
\leq&(2+\sqrt{N})\parallel u\parallel_{L^{3}(\Omega)}\parallel \nabla v\parallel_{L^{6}(\Omega)}\cdot\left\{\int_{\Omega}|\nabla v|^{2}|D^{2}v|^{2}\right\}^{\frac{1}{2}}\nonumber\\
\leq&(2+\sqrt{N})(4+\sqrt{N})^{\frac{1}{3}}\parallel u\parallel_{L^{3}(\Omega)}\parallel v\parallel^{\frac{1}{3}}_{L^{\infty}(\Omega)}\cdot\left\{\int_{\Omega}|\nabla v|^{2}|D^{2}v|^{2}\right\}^{\frac{2}{3}}\nonumber\\
\leq&(2+\sqrt{N})(4+\sqrt{N})^{\frac{1}{3}}C_1^{\frac{1}{3}}\parallel u\parallel_{L^{3}(\Omega)}\cdot\left\{\int_{\Omega}|\nabla v|^{2}|D^{2}v|^{2}\right\}^{\frac{2}{3}}
\end{align}
for all $t\in(0,T_{max})$.
Using Young's inequality, we obtain that
\begin{align}\label{640}
&\int_{\Omega}|\nabla v|^{2}\nabla v\cdot\nabla\cdot(\nabla v\cdot\nabla w)-\int_{\Omega}v\Delta w\nabla\cdot(|\nabla v|^{2}\nabla v)+(2+\sqrt{N})\int_{\Omega}u|\nabla v|^{2}|D^{2}v|\nonumber\\
\leq&\left\{\frac{1}{2}\int_{\Omega}|\nabla v|^{2}|D^{2}v|^{2} \right\}^{\frac{2}{3}}\cdot\left\{2^{\frac{2}{3}}C_3\parallel u\parallel_{L^{3}(\Omega)}\right\}\nonumber\\
\leq&\frac{1}{2}\int_{\Omega}|\nabla v|^{2}|D^{2}v|^{2}+\left\{2^{\frac{2}{3}}C_3\parallel u\parallel_{L^{3}(\Omega)}\right\}^{3}\nonumber\\
=&\frac{1}{2}\int_{\Omega}|\nabla v|^{2}|D^{2}v|^{2}+4C_3^{3}\parallel u\parallel^{3}_{L^{3}(\Omega)}\quad\mbox{for all}~~t\in(0,T_{max}),
\end{align}
where
\begin{align}\label{640987654}
C_3=&\big(1+\frac{\sqrt{N}}{4}\big)(4+\sqrt{N})^{\frac{4}{3}}C_1^{\frac{4}{3}}C_2+(2+\sqrt{N})(4+\sqrt{N})^{\frac{1}{3}}\sqrt{N}
C_1^{\frac{4}{3}}C_2\nonumber\\
&+(2+\sqrt{N})(4+\sqrt{N})^{\frac{1}{3}}C_1^{\frac{1}{3}}.
\end{align}
Substituting \eqref{640} into \eqref{60}, the result of Lemma \ref{Le3.930112} is obtained.
\end{proof}

In order to prove our main theorem, we also need the following estimates of $u$ and $v$.
\begin{lemma}\label{Le3.30112}
Assume that the initial data satisfies \eqref{origin2}. Let $N=3$, $\alpha>2$, for any $\mu>0$ or $\alpha=2$, $\mu>\max\{2|\xi|,2C_3A\}$, $C_3$ is the same as \eqref{640987654} and $A$ is described in the following proof, then there exists a constant $C>0$ such that
\begin{equation}\label{69}
\int_{\Omega}u^{2}(\cdot,t)\leq C  \quad\mbox{for all}~~t\in(0,T_{max})
\end{equation}
and
\begin{equation}\label{70}
\int_{\Omega}|\nabla v(\cdot,t)|^{4}\leq C  \quad\mbox{for all}~~t\in(0,T_{max}).
\end{equation}
\end{lemma}
\begin{proof}
According to Lemma \ref{Le3.3010}, Lemma \ref{Le3.301121} and Lemma \ref{Le3.930112}, if the initial data $(u_0, v_0)$ satisfies \eqref{origin2}, for $N=3$, there are some positive constants $C_i$ $(i=1,2,3)$ such that
\begin{equation}\label{71}
\parallel v\parallel_{L^{\infty}(\Omega)}\leq C_1\quad\mbox{for all}~~t\in(0,T_{max})
\end{equation}
and
\begin{equation}\label{972}
\frac{d}{dt}\int_{\Omega}u^{2}+\int_{\Omega}|\nabla u|^{2}+\frac{\mu}{2}\int_{\Omega}u^{\alpha+1}\leq C_2\int_{\Omega}|\nabla v|^{6}+C_2
 \quad\mbox{for all}~~t\in(0,T_{max})
\end{equation}
as well as
\begin{equation}\label{73}
\frac{d}{dt}\int_{\Omega}|\nabla v|^{4}+2\int_{\Omega}|\nabla v|^{2}|D^{2} v|^{2}+4\int_{\Omega}|\nabla v|^{4} \leq C_3 \int_{\Omega}u^{3} \quad\mbox{for all}~~t\in(0,T_{max}).
\end{equation}
Using Lemma 4.2 in Tao and Winkler \cite{Tao Winkler 2021 NA}, we have
\begin{align}\label{76}
C_2\int_{\Omega}|\nabla v|^{6}\leq&(4+\sqrt{N})^{2}C_2\parallel v\parallel^{2}_{L^{\infty}(\Omega)}\int_{\Omega}|\nabla v|^{2}|D^{2} v|^{2}\nonumber\\
\leq&(4+\sqrt{N})^{2}C_1^{2}C_2\int_{\Omega}|\nabla v|^{2}|D^{2} v|^{2}\quad\mbox{for all}~~t\in(0,T_{max}),
\end{align}
which implies that
\begin{align}\label{7611}
\int_{\Omega}|\nabla v|^{2}|D^{2} v|^{2}
\geq\frac{C_2}{A}\int_{\Omega}|\nabla v|^{6}\quad\mbox{for all}~~t\in(0,T_{max})
\end{align}
with $A:=(4+\sqrt{N})^{2}C_1^{2}C_2$.

Then, $A\times$\eqref{73}+ \eqref{972} implies that
\begin{align}\label{7312}
&A\frac{d}{dt}\int_{\Omega}|\nabla v|^{4}+\frac{d}{dt}\int_{\Omega}u^{2}+\int_{\Omega}|\nabla u|^{2}+\frac{\mu}{2}\int_{\Omega}u^{\alpha+1}-C_3A\int_{\Omega}u^{3}+C_2\int_{\Omega}|\nabla v|^{6}+4A\int_{\Omega}|\nabla v|^{4}\nonumber\\
\leq& C_2\quad\mbox{for all}~~t\in(0,T_{max}).
\end{align}
Specifically, if $\alpha=2$, one can choose $\mu>2C_3A$.
Let
\begin{align}\label{77}
y(t):=\int_{\Omega}u^{2}+A\int_{\Omega}|\nabla v|^{4} ~~t\in(0,T_{max}),
\end{align}
whence \eqref{7312} shows that
\begin{align}\label{78}
&y'(t)+4y(t)\nonumber\\
\leq&\left\{-\frac{\mu}{2}\int_{\Omega}u^{3}-\int_{\Omega}|\nabla u|^{2}+(4+\sqrt{N})^{2}C_1^{2}C_2\int_{\Omega}|\nabla v|^{2}|D^{2} v|^{2}+C_2\right\}\nonumber\\
&+\left\{-2A\int_{\Omega}|\nabla v|^{2}|D^{2} v|^{2}-4A\int_{\Omega}|\nabla v|^{4} +C_3A \int_{\Omega}u^{3}\right\}\nonumber\\
&+4\left\{\int_{\Omega}u^{2}+A\int_{\Omega}|\nabla v|^{4}\right\}\nonumber\\
\leq&\left\{-\frac{\mu}{2}+AC_3\right\}\cdot\int_{\Omega}u^{3}+4\int_{\Omega}u^{2}+C_2+\left\{(4+\sqrt{N})^{2}C_1^{2}C_2-2A\right\}
\cdot\int_{\Omega}|\nabla v|^{2}|D^{2} v|^{2}\nonumber\\
\leq&-\frac{\mu}{4}\int_{\Omega}u^{3}+4\int_{\Omega}u^{2}+C_2\quad\mbox{for all}~~t\in(0,T_{max}).
\end{align}
By the Young's inequality, we can obtain
\begin{align}\label{79}
4u^2=\left\{\frac{\mu}{4}u^3\right\}^{\frac{2}{3}}\cdot\left\{4\cdot(\frac{4}{\mu})^{\frac{2}{3}}\right\}\leq\frac{\mu}{4}u^3
+\left\{4\cdot(\frac{4}{\mu})^{\frac{2}{3}}\right\}^{3}=\frac{\mu}{4}u^3+\frac{1024}{\mu^{2}}
\end{align}
in $\Omega\times(0,T_{max})$. Inserting \eqref{79} into \eqref{78}, we can deduce that
\begin{align}\label{711}
y'(t)+4y(t)\leq C_4 \quad\mbox{for all}~~t\in(0,T_{max})
\end{align}
with $C_4=\frac{1024|\Omega|}{\mu^{2}}+C_2 $. As a consequence of an ODE comparison, we can infer that
\begin{align}\label{712}
y(t)\leq\max\left\{\int_{\Omega}u_0^{2}+A\int_{\Omega}|\nabla v_0|^{4},\frac{C_4}{4}\right\}\quad\mbox{for all}~~t\in(0,T_{max}).
\end{align}
Similarly, if $\alpha>2$, there are constants $C_4>0$ and $C_5>0$ such that
\begin{align}\label{713}
y'(t)+4y(t)\leq&\left\{-\frac{\mu}{2}\int_{\Omega}u^{\alpha+1}+(4+\sqrt{N})^{2}C_1^{2}C_2\int_{\Omega}|\nabla v|^{2}|D^{2} v|^{2}+C_2\right\}\nonumber\\
&+\left\{-2A\int_{\Omega}|\nabla v|^{2}|D^{2} v|^{2}-4A\int_{\Omega}|\nabla v|^{4} +C_3A \int_{\Omega}u^{3}\right\}\nonumber\\
&+4\left\{\int_{\Omega}u^{2}+A\int_{\Omega}|\nabla v|^{4}\right\}\nonumber\\
\leq&-\frac{\mu}{4}\cdot\int_{\Omega}u^{\alpha+1}+4\int_{\Omega}u^{2}+C_5+\left\{(4+\sqrt{N})^{2}C_1^{2}C_2-2A\right\}
\cdot\int_{\Omega}|\nabla v|^{2}|D^{2} v|^{2}\nonumber\\
\leq&-\frac{\mu}{8}\int_{\Omega}u^{\alpha+1}+C_6\nonumber\\
\leq& C_6\quad\mbox{for all}~~t\in(0,T_{max}).
\end{align}
Through the standard ODE comparison arguments, we have
\begin{align}\label{714}
y(t)\leq\max\left\{\int_{\Omega}u_0^{2}+A\int_{\Omega}|\nabla v_0|^{4},\frac{C_6}{4}\right\}\quad\mbox{for all}~~t\in(0,T_{max}).
\end{align}
Thus, the proof of this lemma is finished.
\end{proof}

Based on the above estimates and using the standard elliptic regularity theory again, we can get the estimates as below for $N=3$.
\begin{lemma}\label{Le3.113}
Let $N=3$, $\alpha>2$, for any $\mu>0$ or $\alpha=2$, $\mu>\max\{2|\xi|,2C_3A\}$, $C_3$ and $A$ is the same as in Lemma \ref{Le3.30112}. Then there exists a constant $C>0$ such that
\begin{equation}\label{800}
\parallel w(\cdot,t)\parallel_{W^{1,4}(\Omega)}\leq C \quad\mbox{for all}~~t\in(0,T_{max}).
\end{equation}
Moreover, we have
\begin{equation}\label{800991}
\parallel w(\cdot,t)\parallel_{L^{\infty}(\Omega)}\leq C \quad\mbox{for all}~~t\in(0,T_{max}).
\end{equation}
\end{lemma}
\begin{proof}
By standard elliptic regularity theory, there is a positive constant $C_1$ such that
\begin{equation}\label{801}
\parallel w\parallel_{W^{2,2}(\Omega)}\leq C_1\parallel-\Delta w+w\parallel_{L^2(\Omega)}\quad\mbox{for all}~~w\in W^{2,2}(\Omega) \quad\mbox {such that} ~~\frac{\partial w}{\partial\nu}\big|_{\partial \Omega}=0.
\end{equation}
Since $W^{2,2}(\Omega)\hookrightarrow W^{1,4}(\Omega)$ with $N=3$, hence there is a constant $C_2>0$ such that
\begin{equation}\label{802}
\parallel w\parallel_{W^{1,4}(\Omega)}\leq C_2\parallel w\parallel_{W^{2,2}(\Omega)}\quad\mbox{for all}~~w\in W^{2,2}(\Omega) .
\end{equation}
Recalling the third equation in \eqref{origin} and Lemma \ref{Le3.30112}, we can pick a constant $C_3>0$ fulfilling
\begin{equation}\label{803}
\parallel w\parallel_{W^{1,4}(\Omega)}\leq C_1C_2\parallel-\Delta w+w\parallel_{L^2(\Omega)}\leq C_1C_2\parallel u\parallel _{L^2(\Omega)}\leq C_3 \quad\mbox{for all}~~t\in(0,T_{max}).
\end{equation}
Along with the Sobolev embedding can draw that
\begin{equation}\label{803999}
\parallel w\parallel_{L^{\infty}(\Omega)}\leq C_4\quad\mbox{for all}~~t\in(0,T_{max})
\end{equation}
with some positive constant $C_4>0$. Then we complete the proof of this lemma.
\end{proof}

Next, we consider the case $N=2$.
\begin{lemma}\label{Le003.13}
Let $N=2$, $\alpha=2$, for any $\mu>0$ and $t\in(0,T_{max})$, there exist positive constants $C>0$ and $M>0$ such that
\begin{align}\label{Le003.3111110}
&\frac{d}{dt}\int_{\Omega}(u+1)\ln(u+1)+\int_{\Omega}\frac{|\nabla u|^2}{u+1}+\mu\int_{\Omega}u^2\ln(u+1)\nonumber\\
\leq&M\int_{\Omega}u^2+\frac{1}{4M}\int_{\Omega}|\Delta v|^2+C\int_{\Omega}u^2+a\int_{\Omega}u\ln(u+1)+C.
\end{align}
\end{lemma}
\begin{proof}
Based on the equality \eqref{0L01} and the first equation in \eqref{origin}, we can obtain
\begin{align}\label{Le003.3111111}
&\frac{d}{dt}\int_{\Omega}(u+1)\ln(u+1)\nonumber\\
=&\int_{\Omega}u_{t}\ln(u+1)+\frac{d}{dt}\int_{\Omega}u\nonumber\\
=&\int_{\Omega}u_{t}\ln(u+1)+\int_{\Omega}au-\mu u^2\nonumber\\
\leq&\int_{\Omega}\ln(u+1)(\Delta u-\nabla\cdot(u\nabla v)+\xi\nabla\cdot(u\nabla w)+au-\mu u^{2})-\frac{\mu}{2}\int_{\Omega}u^2+C_1\nonumber\\
=&-\int_{\Omega}\frac{|\nabla u|^2}{u+1}+\int_{\Omega}\frac{u}{u+1}\nabla u\cdot \nabla v-\xi\int_{\Omega}\frac{u}{u+1}\nabla u\cdot\nabla w+a\int_{\Omega}u\ln(u+1)\nonumber\\
&-\mu\int_{\Omega}u^2\ln(u+1)-\frac{\mu}{2}\int_{\Omega}u^2+C_1\quad\mbox{for all}~~t\in(0,T_{max}).
\end{align}
For the  second term on the right hand side of \eqref{Le003.3111111}, with the help of Young' inequality, for any $M>0$, we have
\begin{align}\label{Le003.3111112}
\int_{\Omega}\frac{u}{u+1}\nabla u \cdot\nabla vdx&=\int_{\Omega}\nabla\cdot\int^u_0\frac{\tau}{\tau+1}d\tau \nabla vdx\nonumber\\
&=-\int_{\Omega}\int^u_0\frac{\tau}{\tau+1}d\tau \Delta v dx\nonumber\\
&\leq\int_{\Omega}|\int^u_0\frac{\tau}{\tau+1}d\tau|\cdot |\Delta v| dx\nonumber\\
&\leq\int_{\Omega}u|\Delta v| dx\nonumber\\
&\leq M\int_{\Omega}u^2dx+\frac{1}{4M}\int_{\Omega}|\Delta v|^2dx \quad\mbox{for all}~~t\in(0,T_{max}).
\end{align}
Concerning the third term presented on the right-hand side of \eqref{Le003.3111111} and using the third equation in \eqref{origin}. It follows from Lemma \ref{Le030.2}, we also have
\begin{align}\label{Le003.3111113}
-\xi\int_{\Omega}\frac{u}{u+1}\nabla u\cdot\nabla wdx
&=-\xi\int_{\Omega}\nabla\cdot\int^u_0\frac{\tau}{\tau+1}d\tau \nabla wdx\nonumber\\
&=\xi\int_{\Omega}\int^u_0\frac{\tau}{\tau+1}d\tau \Delta wdx\nonumber\\
&=\xi\int_{\Omega}\frac{\tau}{\tau+1}d\tau wdx-\xi\int_{\Omega}\frac{\tau}{\tau+1}d\tau udx\nonumber\\
&\leq|\xi|\int_{\Omega}uwdx+|\xi|\int_{\Omega}u^{2}dx\nonumber\\
&\leq\frac{|\xi|}{2}\int_{\Omega}u^{2}dx+\frac{|\xi|}{2}\int_{\Omega}w^{2}dx+|\xi|\int_{\Omega}u^{2}dx\nonumber\\
&\leq\frac{3|\xi|}{2}\int_{\Omega}u^{2}dx+C_{2}\quad\mbox{for all}~~t\in(0,T_{max})
\end{align}
with $C_2>0$.
Inserting \eqref{Le003.3111112} and \eqref{Le003.3111113} into \eqref{Le003.3111111}, we have
\begin{align}\label{Le003.3111117}
&\frac{d}{dt}\int_{\Omega}(u+1)\ln(u+1)+\int_{\Omega}\frac{|\nabla u|^2}{u+1}+\mu\int_{\Omega}u^2\ln(u+1)\nonumber\\
\leq& M\int_{\Omega}u^2dx+\frac{1}{4M}\int_{\Omega}|\Delta v|^2dx+C_3\int_{\Omega}u^2+a\int_{\Omega}u\ln(u+1)+C_4
\end{align}
for all $t\in(0,T_{max})$ and $C_i>0$ $(i=3,4)$.
This claim immediately the desired results in Lemma \ref{Le003.13}.
\end{proof}

Applying standard testing procedures to the second equation in system \eqref{origin} again, we can establish the main inequality for $N=2$.
\begin{lemma}\label{Le003.1000300}
Let $N=2$. Then there exists a positive constant $C$ such that
\begin{align}\label{Le003.3111118}
&\frac{d}{dt}\int_{\Omega}|\nabla v|^2+\frac{1}{4}\int_{\Omega}|\Delta v|^2+\int_{\Omega}|\nabla v|^2\leq C\int_{\Omega}u^{2}+C\quad\mbox{for all}~~t\in(0,T_{max}).
\end{align}
\end{lemma}
\begin{proof}
We multiply the second equation in \eqref{origin} by $-\Delta v$ and integrate by parts over $\Omega$
\begin{align}\label{Le003.3111119}
&\frac{d}{dt}\int_{\Omega}|\nabla v|^2+\int_{\Omega}|\Delta v|^2\nonumber\\
=&\int_{\Omega}(\nabla v\cdot\nabla w+v\Delta w)(-\Delta v)-\int_{\Omega}|\nabla v|^2+\int_{\Omega}u(-\Delta v)\\
=&\int_{\Omega}(-\Delta v)(\nabla v\cdot\nabla w)-\int_{\Omega}v\Delta w\cdot\Delta v-\int_{\Omega}|\nabla v|^2+\int_{\Omega}u(-\Delta v)\nonumber \quad\mbox{for all}~~t\in(0,T_{max}).
\end{align}
For the  first term on the right hand side of \eqref{Le003.3111119}, using the Young's inequality again, for any $0<\varepsilon_1<1$, there are positive constants $C_1$ and $C_2$ such that
\begin{align}\label{Le003.3111120}
&\int_{\Omega}(-\Delta v)(\nabla v\cdot\nabla w)dx\nonumber\\
=&-\int_{\Omega}\partial_iw\partial_iv\partial_{kk}vdx\nonumber\\
=&-\int_{\Omega}\partial_iw\partial_ivd\partial_kvdx\nonumber\\
=&\int_{\Omega}\partial_kv(\partial_{ki}w\partial_iv+\partial_iw\partial_{ki}v)dx\nonumber\\
=&\int_{\Omega}\partial_kv\partial_{ki}w\partial_ivdx+\int_{\Omega}\partial_kv\partial_iwd\partial_kv\nonumber\\
=&\int_{\Omega}\partial_kv\partial_{ki}w\partial_ivdx+\frac{1}{2}\int_{\Omega}\partial_iwd|\nabla v|^2\nonumber\\
=&\int_{\Omega}\partial_kv\partial_{ki}w\partial_ivdx-\frac{1}{2}\int_{\Omega}|\nabla v|^2\cdot\Delta wdx\nonumber\\
\leq&\varepsilon_1\int_{\Omega}|\nabla v|^4dx+C_{1}\int_{\Omega}|\Delta w|^2dx+\varepsilon_1\int_{\Omega}|\nabla v|^4dx+C_2\int_{\Omega}|\Delta w|^2dx
\end{align}
for all $t\in(0,T_{max})$. We can infer from the Gagliardo-Nirenberg interpolation inequality and Lemma \ref{Le3.3010}, there exist constants $C_3>0$ and $C_4>0$ such that
\begin{align}\label{Le003.3111121}
2\varepsilon_1\int_{\Omega}|\nabla v|^4\leq& 2\varepsilon_1C_3\big(\parallel\Delta v\parallel^2_{L^2(\Omega)}\parallel v\parallel^2_{L^\infty(\Omega)}+\parallel v\parallel^4_{L^\infty(\Omega)}\big)\nonumber\\
\leq&2\varepsilon_1C_4\big(\parallel\Delta v\parallel^2_{L^2(\Omega)}+1\big)\quad\mbox{for all}~~t\in(0,T_{max}).
\end{align}
Let $2\varepsilon_1C_4=\frac{1}{4}$, we get
\begin{align}\label{Le003.31111219999123}
2\varepsilon_1\int_{\Omega}|\nabla v|^4\leq& 2\varepsilon_1C_3\big(\parallel\Delta v\parallel^2_{L^2(\Omega)}\parallel v\parallel^2_{L^\infty(\Omega)}+\parallel v\parallel^4_{L^\infty(\Omega)}\big)\nonumber\\
\leq&\frac{1}{4}\big(\parallel\Delta v\parallel^2_{L^2(\Omega)}+1\big)\quad\mbox{for all}~~t\in(0,T_{max}).
\end{align}
 Recalling $-\Delta w+w=u$,  let us pick $C_5>0$ fulfilling
\begin{align}\label{Le003.3111122}
(C_{1}+C_2)\parallel w\parallel^2_{W^{2,2}(\Omega)}\leq C_5(\parallel u\parallel^2_{L^2(\Omega)}+1)\quad\mbox{for all}~~t\in(0,T_{max}).
\end{align}
Combining \eqref{Le003.3111120}, \eqref{Le003.31111219999123} with \eqref{Le003.3111122}, we can infer the existence of positive constant $C_6>0$ fulfilling
\begin{align}\label{Le003.3111123}
\int_{\Omega}(-\Delta v)(\nabla v\cdot\nabla w)dx\leq\frac{1}{4}\int_{\Omega}|\Delta v|^2+C_5\int_{\Omega}u^2+C_6\quad\mbox{for all}~~t\in(0,T_{max}).
\end{align}
Similarly, for the second term on the right hand side of \eqref{Le003.3111119}, by using Lemma \ref{Le3.3010}, the $L^{p}$ theory of elliptic equation and Young's inequality, we have
\begin{align}\label{Le003.3111124}
-\int_{\Omega}v\Delta w\cdot\Delta v&\leq \frac{1}{4}\int_{\Omega}|\Delta v|^2+\int_{\Omega}v^2|\Delta w|^2\nonumber\\
&\leq \frac{1}{4}\int_{\Omega}|\Delta v|^2+\parallel v\parallel^2_{L^{\infty}(\Omega)}\parallel \Delta w\parallel^2_{L^{2}(\Omega)}\nonumber\\
&\leq \frac{1}{4}\int_{\Omega}|\Delta v|^2+C_7\int_{\Omega}|\Delta w|^2\nonumber\\
&\leq \frac{1}{4}\int_{\Omega}|\Delta v|^2+C_8\int_{\Omega}u^2 \quad\mbox{for all}~~t\in(0,T_{max})
\end{align}
with $C_7>0,C_8>0$ are positive constants.

Next, regarding the final term presented on the right-hand side of \eqref{Le003.3111119}, due to the Young's inequality, then
\begin{align}\label{Le003.3111125}
\int_{\Omega}u(-\Delta v)&\leq \frac{1}{4}\int_{\Omega}|\Delta v|^2+\int_{\Omega}u^2dx\quad\mbox{for all}~~t\in(0,T_{max}).
\end{align}
Finally, inserting \eqref{Le003.3111123}-\eqref{Le003.3111125} into \eqref{Le003.3111119}, the proof of Lemma \ref{Le003.1000300} is completed.
\end{proof}

We also need to get the following estimates.
\begin{lemma}\label{Le003.1000301}
Let $N=2$, $\alpha=2$, for any $\mu>0$. There exists a positive constant $C$ such that
\begin{align}\label{Le003.3111127}
\int_{\Omega}(u(\cdot,t)+1)\ln(u(\cdot,t)+1)\leq C\quad\mbox{for all}~~t\in(0,T_{max})
\end{align}
and
\begin{align}\label{Le003.3111128}
\int_{\Omega}|\Delta v(\cdot,t)|^2\leq C\quad\mbox{for all}~~t\in(0,T_{max})
\end{align}
as well as
\begin{align}\label{Le003.311112008}
\parallel w(\cdot,t)\parallel_{L^{\infty}(\Omega)}\leq C\quad\mbox{for all}~~t\in(0,T_{max}).
\end{align}

\end{lemma}
\begin{proof}
According to Lemma \ref{Le003.1000300} and Lemma \ref{Le003.13}. Let $M=2$ and $C_i>0$ $ (i=1,2,3,4)$, then we have
\begin{align}\label{Le003.3111129}
&\frac{d}{dt}\int_{\Omega}(u+1)\ln(u+1)+\int_{\Omega}\frac{|\nabla u|^2}{u+1}+\mu\int_{\Omega}u^2\ln(u+1)\nonumber\\
\leq&M\int_{\Omega}u^2+\frac{1}{4M}\int_{\Omega}|\Delta v|^2+C_1\int_{\Omega}u^2+a\int_{\Omega}u\ln(u+1)+C_2\\
=&2\int_{\Omega}u^2+\frac{1}{8}\int_{\Omega}|\Delta v|^2+C_1\int_{\Omega}u^2+a\int_{\Omega}u\ln(u+1)+C_2\nonumber \quad\mbox{for all}~~t\in(0,T_{max})
\end{align}
and
\begin{align}\label{Le003.3111130}
&\frac{d}{dt}\int_{\Omega}|\nabla v|^2+\frac{1}{4}\int_{\Omega}|\Delta v|^2+\int_{\Omega}|\nabla v|^2\leq C_3\int_{\Omega}u^{2}+C_4\quad\mbox{for all}~~t\in(0,T_{max}).
\end{align}
Now \eqref{Le003.3111129}+\eqref{Le003.3111130}, one can find constants $C_5>0$ and $C_6>0$ such that
\begin{align}\label{Le003.3111131}
&\frac{d}{dt}\int_{\Omega}(u+1)\ln(u+1)+\int_{\Omega}\frac{|\nabla u|^2}{u+1}+\frac{d}{dt}\int_{\Omega}|\nabla v|^2\nonumber\\
&+\mu\int_{\Omega}u^2\ln(u+1)+\frac{1}{8}\int_{\Omega}|\Delta v|^2+\int_{\Omega}|\nabla v|^2\nonumber\\
\leq& C_5 \int_{\Omega}u^2+a\int_{\Omega}u\ln(u+1)+C_6\quad\mbox{for all}~~t\in(0,T_{max}).
\end{align}
 Then adding $\int_{\Omega}(u+1)\ln(u+1)$ at the both sides of \eqref{Le003.3111131}, we have
\begin{align}\label{Le003.3111132}
&\frac{d}{dt}\int_{\Omega}(u+1)\ln(u+1)+\int_{\Omega}\frac{|\nabla u|^2}{u+1}+\frac{d}{dt}\int_{\Omega}|\nabla v|^2\nonumber\\
&+\mu\int_{\Omega}u^2\ln(u+1)+\frac{1}{8}\int_{\Omega}|\Delta v|^2+\int_{\Omega}|\nabla v|^2+\int_{\Omega}(u+1)\ln(u+1)\nonumber\\
\leq& C_5 \int_{\Omega}u^2+a\int_{\Omega}u\ln(u+1)+\int_{\Omega}(u+1)\ln(u+1)+C_6\quad\mbox{for all}~~t\in(0,T_{max}).
\end{align}
Since
\begin{equation*}
\lim_{u\rightarrow\infty} \frac{-\mu u^2\ln(u+1)}{C_5 u^2+au\ln(u+1)+(u+1)\ln(u+1)} = -\infty,
\end{equation*}
which implies that there is a positive constant $C_7$ filling
\begin{equation*}
-\mu\int_{\Omega}u^2\ln(u+1)+ C_5 \int_{\Omega}u^2+a\int_{\Omega}u\ln(u+1)+\int_{\Omega}(u+1)\ln(u+1)\leq C_7
\end{equation*}
for all $t\in(0,T_{max})$. Therefore, \eqref{Le003.3111132} can be written as
\begin{align}\label{Le003.3111134}
\frac{d}{dt}\int_{\Omega}(u+1)\ln(u+1)+\frac{d}{dt}\int_{\Omega}|\nabla v|^2+\int_{\Omega}(u+1)\ln(u+1)+\int_{\Omega}|\nabla v|^2\leq C_7
\end{align}
for all $t\in(0,T_{max})$. Let
$$y(t)=\int_{\Omega}(u+1)\ln(u+1)+\int_{\Omega}|\nabla v|^2~~\quad\mbox{for all}~~t\in(0,T_{max}),$$
 which together with \eqref{Le003.3111134} can yield that
\begin{equation*}
y'(t)+y(t)\leq C_7\quad\mbox{for all}~~t\in(0,T_{max}).
\end{equation*}
As a consequence of an ODE comparison, we infer that
\begin{equation*}
y(t)\leq \max\left\{\int_{\Omega}(u_0+1)\ln(u_0+1)+\int_{\Omega}|\nabla v_0|^2, C_7\right\} \quad\mbox{for all}~~t\in(0,T_{max}).
\end{equation*}
Moreover, there exists a constant $C_8>0$ such that
\begin{equation*}
\int_{\Omega}(u+1)\ln(u+1)\leq C_8\quad\mbox{for all}~~t\in(0,T_{max}).
\end{equation*}
Thanks to the result in \cite{Tao Winkler 2014 JDE}, there exists a constant $C_9>0$ such that
\begin{align}\label{Le003.3111132909}
\parallel w\parallel_{L^{\infty}(\Omega)}\leq C_9\quad\mbox{for all}~~t\in(0,T_{max}).
\end{align}
So, the proof of Lemma \ref{Le003.1000301} is finished.
\end{proof}

In following subsection, when $N=2$, $\alpha=2$, for any $\mu>0$, we will provide an arbitrary $L^{p}$-$L^{q}$ estimates of $u$.
\begin{lemma}\label{Le3.13}
Let $N=2$, $\alpha=2$, for any $\mu>0$. Without any restriction on the index $\xi$, then there exists a positive constant $C>0$ such that
\begin{equation}\label{Le3.311}
\parallel u(\cdot,t)\parallel_{L^{p}(\Omega)}\leq C \quad\mbox{for all}~~t\in(0,T_{max}).
\end{equation}
\end{lemma}
\begin{proof}
Testing the first equation in \eqref{origin} by $u^{p-1}$ and noting that $\Delta w=w-u$, we get
\begin{align}\label{Le0040}
&\frac{1}{p}\frac{d}{dt}\int_{\Omega}u^{p}+(p-1)\int_{\Omega}u^{p-2}|\nabla u|^{2}\nonumber\\
=&-\int_{\Omega}\nabla \cdot (u\nabla v)u^{p-1} +\frac{\xi(p-1)}{p}\int_{\Omega}u^{p}\Delta w+ a\int_{\Omega}u^{p}-\mu\int_{\Omega}u^{p+1}\nonumber\\
=&-\frac{(p-1)}{p}\int_{\Omega}u^{p}\Delta v +\frac{\xi(p-1)}{p}\int_{\Omega}u^{p}(w-u)+ a\int_{\Omega}u^{p}-\mu\int_{\Omega}u^{p+1}\nonumber\\
\leq&\frac{(p-1)}{p}\int_{\Omega}u^{p}|\Delta v| +\frac{|\xi|(p-1)}{p}\int_{\Omega}u^{p}w+\frac{|\xi|(p-1)}{p}\int_{\Omega}u^{p+1} + a\int_{\Omega}u^{p}-\mu\int_{\Omega}u^{p+1}
\end{align}
for all $t\in(0,T_{max})$. That is,
\begin{align}\label{Le0041}
&\frac{1}{p}\frac{d}{dt}\int_{\Omega}u^{p}+\frac{p+1}{p}\int_{\Omega}u^{p}\nonumber\\
\leq&\frac{(p-1)}{p}\int_{\Omega}u^{p}|\Delta v| +\frac{|\xi|(p-1)}{p}\int_{\Omega}u^{p}w+\frac{|\xi|(p-1)}{p}\int_{\Omega}u^{p+1}\nonumber\\
&+\frac{p+1}{p}\int_{\Omega}u^{p} +a\int_{\Omega}u^{p}-\mu\int_{\Omega}u^{p+1}\quad\mbox{for all}~~t\in(0,T_{max}).
\end{align}
Then we can invoke the Young's inequality and choose $\varepsilon_1$ appropriately small, then it follows that
\begin{align}\label{Le0042}
\frac{p+1}{p}\int_{\Omega}u^{p}+ a\int_{\Omega}u^{p}-\mu\int_{\Omega}u^{p+1}
\leq(\varepsilon_1-\mu)\int_{\Omega}u^{p+1}+C_1(\varepsilon_1,p)
\end{align}
for all $t\in(0,T_{max})$ and where $$C_1(\varepsilon_1,p)=\frac{1}{p+1}\left(\varepsilon_1\frac{p+1}{p}\right)^{-p}\left(a+\frac{p-1}{p}\right)^{p+1}|\Omega|.$$
Let
\begin{equation*}
A_1=\frac{1}{p+1}\left[\frac{p+1}{p}\right]^{-p}\left(\frac{p-1}{p}\right)^{p+1}=\frac{(p-1)^{p+1}}{p(p+1)^{p+1}}
\end{equation*}
and
\begin{align}\label{Le00042}
\lambda_0:=(A_1C_{13}p)^{\frac{1}{p+1}}.
\end{align}
In view of the first term on the right hand in \eqref{Le0041}, we have
\begin{align}\label{Le0043}
&\frac{(p-1)}{p}\int_{\Omega}u^{p}|\Delta v|\nonumber\\
\leq&\lambda_0\int_{\Omega}u^{p+1}+\frac{1}{p+1}\left[\lambda_0\frac{p+1}{p}\right]^{-p}\left(\frac{p-1}{p}\right)^{p+1}\int_{\Omega}
|\Delta v|^{p+1}\nonumber\\
=&\lambda_0\int_{\Omega}u^{p+1}+A_1\lambda_{0}^{-p}\int_{\Omega}
|\Delta v|^{p+1}\quad\mbox{for all}~~t\in(0,T_{max}).
\end{align}
According the estimate \eqref{Le003.3111132909}, there is a constant $\widetilde{C}>0$ such that
\begin{align}\label{Le010243}
\parallel w\parallel_{L^{\infty}(\Omega)}\leq \widetilde{C}\quad\mbox{for all}~~t\in(0,T_{max}).
\end{align}
An application of the Young's inequality and estimate \eqref{Le010243}. By taking $\varepsilon_2$ appropriately small, one may derive the existence of constants $C_2>0$ and $C_3>0$ such that
\begin{align}\label{Le0044}
&\frac{|\xi|(p-1)}{p}\int_{\Omega}u^{p}w\nonumber\\
\leq& \varepsilon_2\int_{\Omega}u^{p+1}+C_{2}\int_{\Omega}w^{p+1}\nonumber\\
\leq&  \varepsilon_2\int_{\Omega}u^{p+1}+C_{3} \quad\mbox{for all}~~t\in(0,T_{max}).
\end{align}
Inserting \eqref{Le0042}, \eqref{Le0043} and \eqref{Le0044} into \eqref{Le0041} can derive that
\begin{align}\label{Le0045}
\frac{1}{p}\frac{d}{dt}\int_{\Omega}u^{p}\leq&(\varepsilon_1+\varepsilon_2+\lambda_0+\frac{|\xi|(p-1)}{p}-\mu)\int_{\Omega}u^{p+1}
+A_1\lambda_{0}^{-p}
\int_{\Omega}|\Delta v|^{p+1}\nonumber\\
&+C_1(\varepsilon_1,p)+C_{3}\quad\mbox{for all}~~t\in(0,T_{max}).
\end{align}
For any $t\in(s_0,T_{max})$ with $s_0$ is the same as in Lemma \ref{Le71}. Employing the variation-of-constants formula to the above inequality, we obtain
\begin{align}\label{Le0046}
& \frac{1}{p}\parallel u(t)\parallel^{p}_{L^p(\Omega)}\nonumber\\
\leq&\frac{1}{p}e^{-(t-s)(p+1)}\parallel u(s_0)\parallel^{p}_{L^p(\Omega)}+
(\varepsilon_1+\varepsilon_2+\lambda_0+\frac{|\xi|(p-1)}{p}-\mu){\int^{t}_{s_0}e^{-(t-s)(p+1)}\int_\Omega} u^{p+1}\nonumber\\
&+A_1\lambda_{0}^{-p}\int^{t}_{s_0}e^{-(t-s)(p+1)}\int_{\Omega}|\Delta v|^{p+1}+(C_1(\varepsilon_1,p)+C_3)\int^{t}_{s_0}e^{-(t-s)(p+1)}\nonumber\\
\leq&(\varepsilon_1+\varepsilon_2+\lambda_0+\frac{|\xi|(p-1)}{p}-\mu){\int^{t}_{s_0}e^{-(t-s)(p+1)}\int_\Omega} u^{p+1}\\
&+A_1\lambda_{0}^{-p}\int^{t}_{s_0}e^{-(t-s)(p+1)}\int_{\Omega}|\Delta v|^{p+1}+C_{4}\nonumber\quad\mbox{for all}~~t\in(0,T_{max}),
\end{align}
where
$$C_4:=\frac{1}{p}e^{-(t-s)(p+1)}\parallel u(s_0)\parallel^{p}_{L^p(\Omega)}+(C_1(\varepsilon_1,p)+C_3)\int^{t}_{s_0}e^{-(t-s)(p+1)}.$$
 Applying Lemma \ref{Le7} can obtain that
\begin{align}\label{Le0047}
&A_1\lambda_{0}^{-p}\int^{t}_{s_0}e^{-(t-s)(p+1)}\int_{\Omega}|\Delta v|^{p+1}\nonumber\\
=&A_1\lambda_{0}^{-p}e^{-(p+1)t}\int^{t}_{s_0}e^{(p+1)s}\int_{\Omega}|\Delta v|^{p+1}\nonumber\\
\leq&A_1\lambda_{0}^{-p}e^{-(p+1)t}C_{p+1}[\int^{t}_{s_0}e^{(p+1)s}\parallel\nabla\cdot(v\nabla w)+u\parallel^{p+1}_{L^{p+1}(\Omega)}\nonumber\\
&+e^{(p+1)s_0}(\parallel v(\cdot,s_0)\parallel^{p+1}_{L^{p+1}(\Omega)}+\parallel \Delta v(\cdot,s_0)\parallel^{p+1}_{L^{p+1}(\Omega)})]\quad\mbox{for all}~~t\in(s_0,T_{max})
\end{align}
with $C_{p+1}>0$. Since
\begin{align}\label{Le0048}
&\parallel\nabla\cdot(v\nabla w)+u\parallel^{p+1}_{L^{p+1}(\Omega)}\nonumber\\
&=\parallel \nabla v\cdot\nabla w+v \Delta w+u\parallel^{p+1}_{L^{p+1}(\Omega)}\nonumber\\
&\leq\parallel v \Delta w\parallel^{p+1}_{L^{p+1}(\Omega)}+\parallel\nabla v\cdot\nabla w\parallel^{p+1}_{L^{p+1}(\Omega)}+\parallel u\parallel^{p+1}_{L^{p+1}(\Omega)}\quad\mbox{for all}~~t\in(0,T_{max}).
\end{align}
For the first term on the right hand side of \eqref{Le0048}, in light of the $L^{p}$ theory of elliptic equation and Lemma \ref{Le3.3010}, we derive that there is $C_5>0$ such that
\begin{align}\label{Le0049}
\parallel v \Delta w\parallel^{p+1}_{L^{p+1}(\Omega)}&\leq\parallel v\parallel^{p+1}_{L^{\infty}(\Omega)}\parallel \Delta w\parallel^{p+1}_{L^{p+1}(\Omega)}\nonumber\\
&\leq C_5\parallel  u\parallel^{p+1}_{L^{p+1}(\Omega)}\quad\mbox{for all}~~t\in(0,T_{max}).
\end{align}
Moreover, for the second term on the right hand side of \eqref{Le0048}, for any $0<\varepsilon<1$, we have
\begin{align}\label{Le00491}
\parallel\nabla v\cdot\nabla w\parallel^{p+1}_{L^{p+1}(\Omega)}&\leq \frac{1}{4\varepsilon}\int_{\Omega} |\nabla w|^{2p+2}+ \varepsilon\int_{\Omega} |\nabla v|^{2p+2}\quad\mbox{for all}~~t\in(0,T_{max}).
\end{align}
The Gagliardo-Nirenberg inequality and Lemma \ref{Le3.3010} implies that
\begin{align}\label{Le00492}
\varepsilon\int_{\Omega} |\nabla v|^{2p+2}&\leq C_{6}\varepsilon \big(\parallel \Delta v\parallel^{p+1}_{L^{p+1}(\Omega)}\parallel v\parallel^{p+1}_{L^{\infty}(\Omega)}+\parallel v\parallel^{2p+2}_{L^{\infty}(\Omega)}\big)\nonumber\\
&\leq C_{6}\varepsilon \int_{\Omega}|\Delta v|^{p+1}+C_7\quad\mbox{for all}~~t\in(0,T_{max})
\end{align}
with $C_7>0$. Let $C_{p+1}C_{6}\varepsilon=\frac{1}{2}$, which implies that $\varepsilon=\frac{1}{2C_{p+1}C_{6}}$. Using the Gagliardo-Nirenberg inequality and $L^{p}$ theory of elliptic equation  as well as \eqref{Le003.311112008}, there are some constants $C_i>0$ $(i=8,9,10,11)$ such that
\begin{align}\label{Le00493}
\frac{1}{4\varepsilon}\parallel \nabla w\parallel^{2p+2}_{L^{2p+2}(\Omega)}
\leq& C_8\left(\parallel \Delta w\parallel^{p+1}_{L^{p+1}(\Omega)}\parallel  w\parallel^{p+1}_{L^{\infty}(\Omega)}+\parallel w\parallel^{2p+2}_{L^{\infty}(\Omega)}\right)\nonumber\\
\leq&C_9\left(\parallel \Delta w\parallel^{p+1}_{L^{p+1}(\Omega)}+1\right)\nonumber\\
\leq&C_{10}\int_{\Omega}u^{p+1}+C_{11}\quad\mbox{for all}~~t\in(0,T_{max}),
\end{align}
which implies that
\begin{align}\label{Le00494}
&A_1\lambda_{0}^{-p}\int^{t}_{s_0}e^{-(t-s)(p+1)}\int_{\Omega}|\Delta v|^{p+1}\nonumber\\
\leq&A_1\lambda_{0}^{-p}e^{-(p+1)t}C_{p+1}\int^{t}_{s_0}e^{(p+1)s}\parallel\nabla\cdot(v\nabla w)+u\parallel^{p+1}_{L^{p+1}(\Omega)}+C_{12}\nonumber\\
\leq&A_1\lambda_{0}^{-p}e^{-(p+1)t}C_{p+1}\int^{t}_{s_0}e^{(p+1)s}\parallel\nabla v\cdot\nabla w\parallel^{p+1}_{L^{p+1}(\Omega)}\nonumber\\
&+A_1\lambda_{0}^{-p}e^{-(p+1)t}C_{p+1}\int^{t}_{s_0}e^{(p+1)s}\parallel v\Delta w\parallel^{p+1}_{L^{p+1}(\Omega)}\nonumber\\
&+A_1\lambda_{0}^{-p}e^{-(p+1)t}C_{p+1}\int^{t}_{s_0}e^{(p+1)s}\int_{\Omega} u^{p+1}+C_{12}\nonumber\\
\leq&\frac{1}{2}A_1\lambda_{0}^{-p}e^{-(p+1)t}\int^{t}_{s_0}e^{(p+1)s}\int_{\Omega}|\Delta v|^{p+1}\nonumber\\
&+A_1\lambda_{0}^{-p}e^{-(p+1)t}C_{p+1}\int^{t}_{s_0}e^{(p+1)s}C_{10}\int_{\Omega} u^{p+1}\nonumber\\
&+A_1\lambda_{0}^{-p}e^{-(p+1)t}C_{p+1}\int^{t}_{s_0}e^{(p+1)s}C_5 \int_{\Omega}u^{p+1}\nonumber\\
&+A_1\lambda_{0}^{-p}e^{-(p+1)t}C_{p+1}\int^{t}_{s_0}e^{(p+1)s}\int_{\Omega} u^{p+1}+C_{12}\nonumber\\
&+A_1\lambda_{0}^{-p}e^{-(p+1)t}C_{p+1}\int^{t}_{s_0}e^{(p+1)s}(C_7+C_{11})\quad\mbox{for all}~~t\in(s_0,T_{max})
\end{align}
with $C_{12}>0$. That is, there are some constants $C_{13}>0$ and $C_{14}>0$ such that
\begin{align}\label{Le00495}
&A_1\lambda_{0}^{-p}\int^{t}_{s_0}e^{-(t-s)(p+1)}\int_{\Omega}|\Delta v|^{p+1}\nonumber\\
\leq&A_1\lambda_{0}^{-p}e^{-(p+1)t}C_{13}\int^{t}_{s_0}e^{(p+1)s}\int_{\Omega} u^{p+1}+C_{14}\quad\mbox{for all}~~t\in(s_0,T_{max}).
\end{align}
Combining \eqref{Le0046} with \eqref{Le00495}, we can obtain that there is a constant $C_{15}>0$ such that
\begin{align}\label{Le00496}
\frac{1}{p}\parallel u(t)\parallel^{p}_{L^p(\Omega)}
\leq&(\varepsilon_1+\varepsilon_2+\lambda_0+\frac{|\xi|(p-1)}{p}-\mu)\int^{t}_{s_0}e^{-(t-s)(p+1)}\int_{\Omega} u^{p+1}\nonumber\\
&+A_1\lambda_{0}^{-p}e^{-(p+1)t}C_{13}\int^{t}_{s_0}e^{(p+1)s}\int_{\Omega} u^{p+1}+C_{15}\\
=&(\varepsilon_1+\varepsilon_2+\lambda_0+\frac{|\xi|(p-1)}{p}-\mu+A_1\lambda_{0}^{-p}C_{13})\int^{t}_{s_0}e^{-(t-s)(p+1)}\int_{\Omega} u^{p+1}+C_{15}.\nonumber
\end{align}
In view of \eqref{Le00042}, we have $$\lambda_{0}=(pA_1C_{13})^{\frac{1}{p+1}}.$$
Now, in light of Lemma \ref{Le20}, let
\begin{align*}
g(p)=&\varepsilon_1+\varepsilon_2+\lambda_0+\frac{|\xi|(p-1)}{p}-\mu+A_1\lambda_{0}^{-p}C_{13}\nonumber\\
=&\varepsilon_1+\varepsilon_2+\frac{|\xi|(p-1)}{p}+\frac{p-1}{p}C^{\frac{1}{p+1}}_{13}-\mu\nonumber\\
\end{align*}
and
\begin{equation*}
h(p)=\frac{|\xi|(p-1)}{p}+\frac{p-1}{p}C^{\frac{1}{p+1}}_{13}-\mu,
\end{equation*}
then
\begin{equation*}
\lim_{p\rightarrow1}h(p)<0.
\end{equation*}
So, there exists $p_0>\frac{N}{2}=1$ and  $1<p<p_0$ such that $h(p_0)<0$. Choosing $\varepsilon_1+\varepsilon_2$ arbitrarily small such that
$$\varepsilon_1+\varepsilon_2+\frac{|\xi|(p-1)}{p}+\frac{p_0-1}{p_0}C^{\frac{1}{p+1}}_{13}-\mu<0.$$
Then there exists a positive constant $C_{16}$ such that
\begin{align}\label{Le00498}
\int_{\Omega} u^{p_0}\leq C_{16} \quad\mbox{for all}~~t\in(0,T_{max}),
\end{align}
which means the result of Lemma \ref{Le3.13} is hold.
\end{proof}

Now, we can use the standard elliptic regularity theory and the standard semigroup arguments to obtain the uniform bound of $v$.
\begin{lemma}\label{Le3.007707}
If $N=2$, $\alpha=2$, for any $\mu>0$ and $N=3$, $\alpha>2$, for any $\mu>0$ or $\alpha=2$,  $\mu>\max\{2|\xi|,2C_3A\}$, where $C_3$ and $A$ is the same as Lemma \ref{Le3.30112}. Let $N<q<\frac{Np}{N-p}$, there exists a positive constant $C$ such that
\begin{align}\label{99999999999}
\parallel \nabla v(\cdot,t) \parallel_{L^{q}(\Omega)} \leq C \quad\mbox{for all}~~t\in(0,T_{max}).
\end{align}
\end{lemma}
\begin{proof}
According to Lemma \ref{Le3.13}, we choose $\widetilde{p_0}$ satisfies $\widetilde{p_0}=p_0$ for $N=2$ and $\widetilde{p_0}=2$ for $N=3$, which guarantees that $\widetilde{p_0}>\frac{N}{2}$.
Then there exists a $C_1>0$ such that
\begin{align}\label{Le00499}
\int_{\Omega} u^{\widetilde{p}_0}\leq C_{1}\quad\mbox{for all}~~t\in(0,T_{max}).
\end{align}
In view of the second equation in \eqref{origin},
\begin{equation}
v=e^{t(\Delta-1)}v_0+\int^{t}_{0}e^{(t-s)(\Delta-1)}\left(\nabla v\cdot\nabla w+v \Delta w+u\right).
\end{equation}
That is,
\begin{align}\label{Le00500001}
\parallel v\parallel_{W^{1,q}}
\leq& C_2+\int^{t}_{t_0}(1+(t-s)^{-\frac{1}{2}-\frac{N}{2}(\frac{1}{\widetilde{p}_0}-\frac{1}{q})})(\parallel \nabla v\cdot\nabla w\parallel_{L^{\widetilde{p}_0}(\Omega)}\nonumber\\
&+\parallel v \Delta w\parallel_{L^{\widetilde{p}_0}(\Omega)}+\parallel u \parallel_{L^{\widetilde{p}_0}(\Omega)})\quad\mbox{for all}~~t\in(0,T_{max})
\end{align}
with some certain positive constant $C_2$. Due to {\eqref{Le00499}}, the H\"{o}lder inequality, Lemma \ref{Le3.3010} and $L^{p}$ theory of elliptic equation, there exist  positive constants $C_i>0$ $(i=3,4)$, we have
\begin{align}\label{Le00500002}
\parallel v\Delta w\parallel_{L^{\widetilde{p}_0}(\Omega)}
\leq& \parallel v\parallel_{L^{\infty}(\Omega)}\parallel \Delta w\parallel_{L^{\widetilde{p}_0}(\Omega)}\nonumber\\
\leq&C_{3}\parallel \Delta w\parallel_{L^{\widetilde{p}_0}(\Omega)}\nonumber\\
\leq&C_{3}(\parallel u\parallel_{L^{\widetilde{p}_0}(\Omega)}+1)\nonumber\\
\leq&C_{4}\quad\mbox{for all}~~t\in(0,T_{max}).
\end{align}
 Using the $L^{p}$ theory of elliptic equation and $\widetilde{p}_0>\frac{N}{2}$, we can get $C_5>0$ and $C_6>0$ such that
\begin{align}\label{Le00500004}
\parallel w\parallel_{W^{1,s}}\leq\parallel w\parallel_{W^{2,\widetilde{p}_0}}\leq C_5\parallel u\parallel_{L^{\widetilde{p}_0}}\leq C_6\quad\mbox{for all}~~t\in(0,T_{max}),
\end{align}
where $N<s<\frac{\widetilde{p}_0N}{N-\widetilde{p}_0}$.
Using the H\"{o}lder inequality again to obtain that
\begin{align}\label{Le00500003}
&\left(\int_{\Omega}|\nabla v\cdot\nabla w|^{\widetilde{p}_0} dx\right)^\frac{1}{\widetilde{p}_0}\nonumber\\
\leq& \left(\int_{\Omega} |\nabla v|^{\widetilde{p}_0\theta'}dx\right)^{\frac{1}{\widetilde{p}_0\theta'}}\left(\int_{\Omega} |\nabla w|^{\widetilde{p}_0\theta}dx\right)^{\frac{1}{\widetilde{p}_0\theta}}\nonumber\\
\leq&C_6\left(\int_{\Omega} |\nabla v|^{\widetilde{p}_0\theta'}dx\right)^{\frac{1}{\widetilde{p}_0\theta'}}\quad\mbox{for all}~~t\in(0,T_{max}),
\end{align}
where $\theta=\frac{s}{\widetilde{p}_0}$ and $\theta'=\frac{s}{s-\widetilde{p}_0}$. Then we can get
\begin{align}\label{Le00500005}
&\left(\int_{\Omega}|\nabla v\cdot\nabla w|^{\widetilde{p}_0}dx\right)^\frac{1}{\widetilde{p}_0}\nonumber\\
\leq&C_6\left(\int_{\Omega} |\nabla v|^{\frac{s\widetilde{p}_0}{s-\widetilde{p}_0}}dx\right)^{\frac{s-\widetilde{p}_0}{s\widetilde{p}_0}}\nonumber\\
=&C_6\parallel \nabla v\parallel_{L^{\frac{s\widetilde{p}_0}{s-\widetilde{p}_0}}(\Omega)}\quad\mbox{for all}~~t\in(0,T_{max}).
\end{align}
Let $M(T):=\max\limits_{0\leq t\leq T_{max}}\parallel \nabla v\parallel_{L^{q}(\Omega)}$, for some $C_7>0$, which implies that
\begin{align}\label{Le00500006}
M(T) \leq C_2+\int^{t}_{t_0}(1+(t-s)^{-\frac{1}{2}-\frac{N}{2}(\frac{1}{\widetilde{p}_0}-\frac{1}{q})})(C_6\parallel \nabla v\parallel_{L^{\frac{s\widetilde{p}_0}{s-\widetilde{p}_0}}(\Omega)}+C_7)ds.
\end{align}
In view of $\widetilde{p}_0>\frac{N}{2}$ and $N<s<\frac{\widetilde{p}_0N}{N-\widetilde{p}_0}$, we have $\frac{\widetilde{p}_0s}{s-\widetilde{p}_0}<q$. Then by the Gagliardo-Nirenberg interpolation inequality, there exist constants $C_8>0$ and $C_9>0$ such that
\begin{align}\label{Le00500007}
C_6\parallel \nabla v\parallel_{L^{\frac{s\widetilde{p}_0}{s-\widetilde{p}_0}}(\Omega)}\leq& C_8(\parallel \nabla v\parallel^{\xi_6}_{L^{q}(\Omega)}\parallel \nabla v\parallel^{1-\xi_6}_{L^{1}(\Omega)}+\parallel \nabla v\parallel_{L^{1}(\Omega)})\nonumber\\
\leq& C_9(\parallel \nabla v\parallel^{\xi_6}_{L^{q}(\Omega)}+1)\quad\mbox{for all}~~t\in(0,T_{max}),
\end{align}
where $\xi_6=\frac{1+\frac{1}{N}-\frac{s-\widetilde{p}_0}{s\widetilde{p}_0}}{1+\frac{1}{N}-\frac{1}{q}}\in(0,1)$. Then, combining \eqref{Le00500007} with \eqref{Le00500006} can obtain that, there are positive constants $C_i>0$ $(i=10,11,12)$ such that
\begin{align*}
M(T) \leq C_2+\int^{t}_{t_0}(1+(t-s)^{-\frac{1}{2}-\frac{N}{2}(\frac{1}{\widetilde{p}_0}-\frac{1}{q})})(\parallel \nabla v\parallel^{\xi_6}_{L^{q}(\Omega)}+C_{10})ds
\leq C_{11}+C_{12}M^{\xi_6}(T).
\end{align*}
Since $\xi_6\in(0,1)$, we further have
\begin{align}\label{Le00500009}
M(T) \leq \max\left\{2C_{11}, (2C_{12})^{\frac{1}{1-\xi_6}}\right\}.
\end{align}
This completes the proof of this Lemma.
\end{proof}

In view of Lemma \ref{Le3.30112}, we can achieve the boundedness of $u$ in $L^{\infty}$ uniformly in $(0, T_{max})$.
\begin{lemma}\label{Le3.9997}
If $N=2$, $\alpha=2$, for any $\mu>0$ and $N=3$, $\alpha>2$, for any $\mu>0$ or $\alpha=2$, $\mu>\max\{2|\xi|,2C_3A\}$, where $C_3$ and $A$ is the same as Lemma \ref{Le3.30112}. Then there exists a positive constant $C$ such that
\begin{equation}
\parallel u(\cdot,t) \parallel_{L^{\infty}(\Omega)} \leq C \quad\mbox{for all}~~t\in(0,T_{max}).
\end{equation}
\end{lemma}
\begin{proof}
Recalling the first equation in \eqref{origin}, we have
$$u_t=\Delta u-\nabla\cdot(u\nabla v)+\xi\nabla\cdot(u\nabla w)+au-\mu u^{\alpha}.$$
 The variation-of-constants formula associated with the above equation can represent the solution $u$ as follows
 \begin{align}\label{Le3.76100}
u(\cdot,t)=&e^{(t-t_{0})\Delta}u(\cdot,t_{0})+\int^{t}_{t_{0}}e^{(t-s)\Delta}\left\{\xi\nabla\cdot(u\nabla w)-\nabla\cdot(u\nabla v)\right\}\nonumber\\
&+\int^{t}_{t_{0}}e^{(t-s)\Delta}(au-\mu u^{\alpha})\quad\mbox{for all}~~t\in(0,T_{max}).
\end{align}
For convenience, let
$$ h:=-\nabla v+\xi\nabla w, ~~f(u):=au-\mu u^{\alpha}.$$
Then, \eqref{Le3.76100} can be rewritten as
\begin{align}\label{Le3.76200}
u(\cdot,t)=e^{(t-t_{0})\Delta}u(\cdot,t_{0})+\int^{t}_{t_{0}}e^{(t-s)\Delta}\nabla\cdot\left[u(\cdot,s)h(\cdot,s)\right]ds
+\int^{t}_{t_{0}}e^{(t-s)\Delta}f(u(\cdot,s))ds
\end{align}
for all $t\in(t_0,T_{max})$, where $t_{0}: =(t-1)_{+}$.
Hence
\begin{align}\label{Le3.75}
\parallel u(\cdot,t)\parallel_{L^{\infty}(\Omega)}
\leq& \parallel e^{(t-t_{0})\Delta}u(\cdot,t_{0})\parallel_{L^{\infty}(\Omega)}\nonumber\\
&+\int^{t}_{t_{0}}\parallel e^{(t-s)\Delta}\nabla\cdot\left[u(\cdot,s)h(\cdot,s)\right]\parallel_{L^{\infty}(\Omega)}ds \nonumber\\ &+\int^{t}_{t_{0}}\parallel e^{(t-s)\Delta}f(u(\cdot,s))\parallel_{L^{\infty}(\Omega)}ds  \quad\mbox{for all}~~t\in(0,T_{max}).
\end{align}
Accordingly, if $t\in(0,1]$, by using the maximum principle, there is a constant $C_1>0$ such that
\begin{equation}
\parallel e^{(t-t_{0})\Delta}u_{0}\parallel_{L^{\infty}(\Omega)}\leq\parallel u_0\parallel_{L^{\infty}(\Omega)}\leq C_1\quad\mbox{for all}~~t\in(0,T_{max}).
\end{equation}
However, if $t>1$, by using an application of the Neumann heat semigroup $(e^{t\Delta})_{t>0}$ on $\Omega$ together with Lemma \ref{Le3.1}, we can yield that there are constants $C_2>0$ and $C_3>0$ such that
\begin{equation}
\parallel e^{(t-t_{0})\Delta}u(\cdot,t_{0})\parallel_{L^{\infty}(\Omega)}\leq C_2(t-t_0)^{-\frac{N}{2}}\parallel u(\cdot,t_0)\parallel_{L^{1}(\Omega)}
\leq C_3\quad\mbox{for all}~~t\in(0,T_{max}).
\end{equation}
In summary, we obtain $C_4>0$ fulfilling
\begin{equation}\label{3.755}
\parallel e^{(t-t_{0})\Delta}u(\cdot,t_{0})\parallel_{L^{\infty}(\Omega)}
\leq C_4\quad\mbox{for all}~~t\in(0,T_{max}).
\end{equation}
Similarly, we can estimate the third integral on the right of \eqref{Le3.75} as follows:
\begin{align}\label{Le3.77}
&\int^{t}_{t_{0}}\parallel e^{(t-s)\Delta}f(u(\cdot,s))\parallel_{L^{\infty}(\Omega)}ds\nonumber\\
\leq& \int^{t}_{t_{0}}sup\left(as-\mu s^{\alpha}\right)_{+}ds \nonumber\\
\leq&\int^{t}_{t_{0}}\frac{a^{\frac{\alpha}{\alpha-1}}_{+}}{\mu^{\frac{1}{\alpha-1}}}\alpha^{\frac{-1}{\alpha-1}}(1-\frac{1}{\alpha})ds
\nonumber\\
\leq& \frac{a^{\frac{\alpha}{\alpha-1}}_{+}}{\mu^{\frac{1}{\alpha-1}}}\alpha^{\frac{-1}{\alpha-1}}(1-\frac{1}{\alpha})\nonumber\\
\leq& C_5\quad\mbox{for all}~~t\in(0,T_{max}),
\end{align}
where we used the fact $0\leq t-t_0\leq1$ and $C_5>0$.
In view of \eqref{Le00499}, we get
\begin{align}\label{Le0050000010}
\int_{\Omega} u^{p_0}\leq C_{6}\quad\mbox{for all}~~t\in(0,T_{max})
\end{align}
with $C_{6}>0$. Now using the the $L^{p}$ theory of elliptic equation can find some constants $C_{7}>0$ and $C_{8}>0$ such that
\begin{align}\label{Le0050000011}
\parallel w\parallel_{W^{2,p_0}(\Omega)}\leq C_7\parallel u\parallel_{L^{p_0}(\Omega)}\leq C_8\quad\mbox{for all}~~t\in(0,T_{max}).
\end{align}
Since $W^{2,p_0}(\Omega)\hookrightarrow W^{1,q}(\Omega)$ for the case of $q>N$, then there exist constants $C_9>0$ and $C_{10}>0$ such that
\begin{align}\label{Le0050000012}
\parallel w\parallel_{W^{1,q}(\Omega)}\leq C_9\parallel w\parallel_{W^{2,p_0}(\Omega)}\leq C_{10}\quad\mbox{for all}~~t\in(0,T_{max}).
\end{align}
Subsequently, from Lemma \ref{Le3.007707} and \eqref{Le0050000012}, we can find a constant $C_{11}>0$ such that
\begin{align}\label{Le3.750}
\parallel h(\cdot,t) \parallel_{L^{q}(\Omega)}
=&\parallel-\nabla v+\xi\nabla w\parallel_{L^{q}(\Omega)}\nonumber\\
\leq&\parallel \nabla v \parallel_{L^{q}(\Omega)}+\parallel \xi\nabla w \parallel_{L^{q}(\Omega)}\nonumber\\
\leq &C_{11}  \quad\mbox{for all}~~t\in(0,T_{max}).
\end{align}
Now recalling Lemma \ref{Le3.1}, we can find $C_{12}>0$ satisfying
\begin{align}\label{Le3.751}
\parallel u(\cdot,t) \parallel_{L^{1}(\Omega)}\leq C_{12}  \quad\mbox{for all}~~t\in(0,T_{max}).
\end{align}
Using the smoothing properties of the Stokes semigroup {and the H\"{o}lder inequality}, if $N<p<q$, then there exist some constants $\lambda_1>0$ as well as $C_{13}>0$ such that
\begin{align}\label{3.752}
&\int^{t}_{t_{0}}\parallel e^{(t-s)\Delta}\nabla\cdot[u(\cdot,s)h(\cdot,s)]\parallel_{L^{\infty}(\Omega)}ds\nonumber\\
\leq& C_{13}\int^{t}_{t_{0}}\left[1+(t-s)^{-\frac{1}{2}-\frac{N}{2}\times\frac{1}{p}}\right]e^{-\lambda_1(t-s)}
\parallel u(\cdot,s)h(\cdot,s)\parallel_{L^{p}(\Omega)} ds\nonumber\\
\leq&C_{13}\int^{t}_{t_{0}}\left[1+(t-s)^{-\frac{1}{2}-\frac{N}{2}\times\frac{1}{p}}\right]e^{-\lambda_1(t-s)}\parallel  u(\cdot,s)\parallel_{L^{\frac{qp}{q-p}}(\Omega)}\parallel h(\cdot,s) \parallel_{L^{q}(\Omega)}ds
\end{align}
for all $t\in(0,T_{max})$. Writing
\begin{equation*}
M(T):=\max_{t\in[0,T]}\parallel u(\cdot,t)\parallel_{L^{\infty}(\Omega)} \quad\mbox{for all}~~t\in(0,T_{max})\ \ \mbox{and}\ \ T\in(0,T_{max}).
\end{equation*}
Relying on \eqref{Le3.751} to see that for any $T>0$,
\begin{align}\label{3.753}
\parallel  u(\cdot,s)\parallel_{L^{\frac{qp}{q-p}}(\Omega)}
\leq \parallel u(\cdot,s)\parallel^{\theta}_{L^{\infty}(\Omega)}\parallel u(\cdot,s)\parallel^{{1-\theta}}_{L^1(\Omega)}
\leq C_{12}^{1-\theta}M^{\theta}(T) \quad\mbox{for all}~~s\in(0,T)
\end{align}
with {$\theta=\frac{pq-q+p}{pq}\in(0,1)$}. Combining {\eqref{Le3.750}}, \eqref{3.752} with \eqref{3.753}, we obtain
 \begin{align}\label{3.754}
&\int^{t}_{t_{0}}\parallel e^{(t-s)\Delta}\nabla\cdot[u(\cdot,s)h(\cdot,s)]\parallel_{L^{\infty}(\Omega)}ds\nonumber\\
\leq&C_{11}C_{13}C_{12}^{1-\theta}M^{\theta}(T)\int^{t}_{t_{0}}\left[1+(t-s)^{-\frac{1}{2}-\frac{N}{2}\times\frac{1}{p}}\right]
e^{-\lambda_1(t-s)}ds\nonumber\\
\leq&C_{11}C_{13}C_{12}^{1-\theta}C_{14}M^{\theta}(T)\quad\mbox{for all}~~t\in(0,T_{max})
\end{align}
 with
 $$C_{14}=\int^{\infty}_{0}[1+(t-s)^{-\frac{1}{2}-\frac{N}{2}\times\frac{1}{p}}]e^{-\lambda_1(t-s)}.$$
Collecting \eqref{Le3.75}, \eqref{3.755}, \eqref{Le3.77} and \eqref{3.754}.
Thus, we conclude that
\begin{equation}
M(T)\leq C_4+C_5+C_{11}C_{13}C_{12}^{1-\theta}C_{14}M^{\theta}(T) \quad\mbox{for all}~~t\in(0,T_{max}).
\end{equation}
Since $\theta<1$, we have
\begin{equation}
M(T)\leq \max \left\{\big(\frac{C_4+C_5}{C_{11}C_{13}C_{12}^{1-\theta}C_{14}}\big)^{\frac{1}{\theta}},(2C_{11}C_{13}C_{12}^{1-\theta}C_{14})
^{\frac{1}{1-\theta}}\right\} \quad\mbox{for all}~~t\in(0,T_{max}).
\end{equation}
Summing up, the proof is completed.
\end{proof}

\section{The proof of Theorem \ref{theorem}}

Now, we are in a position to prove Theorem \ref{theorem}, based on the estimates of solutions for system \eqref{origin}, which were obtained in the above sections.

\textbf{The proof of Theorem 1.1.} In view of the extensibility criterion in Lemma \ref{Le3}, we can assert $T_{max}=\infty$ according to the estimates of Lemma \ref{Le3.7}, Lemma \ref{Le3.3010} and Lemma \ref{Le3.9997}, hence the global boundedness properties in Theorem \ref{theorem} hold.

\section*{Acknowledgment}
 The authors would like to thank the anonymous referees whose comments help to improve the contain of this paper. The authors J. Zheng and K. Li are supported by the Natural Science Foundation of Shandong Province (No.ZR2022JQ06) and the third author K. Li are supported by Natural Science Foundation of China (No.12101534) and the Natural Science Foundation of Shandong Province (No.ZR2021QA052).


\begin{thebibliography}{00}


\bibitem{Mcelwain Balding 1985 JTB}
D. Balding, D. Mcelwain, A mathematical model of tumour induced capillary growth, {\it J. Theor. Biol.}, \textbf{114} (1985), 53--73.

\bibitem{HW1973JMSJ}
H. Br\'{e}zis, W.A. Strauss, Semi-linear second-order elliptic equations in $L^{1}$, {\it J. Math. Soc. Jpn.} , \textbf{25} (1973), 565--590.

\bibitem{cao 2014 JMAA}
X. Cao, Boundedness in a quasilinear parabolic-parabolic Keller-Segel system with logistic source, {\it J. Math. Anal. Appl.}, \textbf{412} (2014), 181--188.

\bibitem{Chaplain Stuart 1993 JMAMB}
M. Chaplain, A. Stuart, A model mechanism for the chemotactic response of endothelial cells to tumour angiogenesis
factor,  {\it IMA J. Math. Appl. Med. Biol.}, \textbf{10} (1993), 149--168.


\bibitem{Ding Wang 2019 DCDSSB}
M. Ding, W. Wang, Global boundedness in a quasilinear fully parabolic chemotaxis system with indirect signal production, {\it Discrete Contin. Dyn. Syst. Ser. B}, \textbf{24} (2019), 4665--4684.


\bibitem{Edelstein 1982 JTB}
L. Edelstein, The propagation of fungal colonies: a model for tissue growth, {\it J. Theor. Biol.} \textbf{98} (1982), 679--701.

\bibitem{FKI 1991 DCDS}
K. Fujie, A. Ito, M. Winkler, T. Yokota, Stabilization in a chemotaxis model for tumor invasion. {\it Discrete Contin. Dyn. Syst.},
 \textbf{36} (2016), 151--169.

\bibitem{FK 1987science}
J. Folkman, M.Klagsbrun, Angiogenic factors, {\it Science}, \textbf{235}(4787) (1987), 442--447.

\bibitem{Fujie 2020 DCDSSS}
K. Fujie, Global asymptotic stability in a chemotaxis-growth model for tumor invasion, {\it Discrete Contin. Dyn. Syst. Ser. S}, \textbf{13} (2020), 203--209.


\bibitem{FA 1995 nature}
J. Folkman, Angiogenesis in cancer, vascular, rheumatoid and other disease, {\it Nature Med.}, \textbf{1} (1995), 27--31.

\bibitem{FA 1950 A.C.R}
J. Folkman, Tumor angiogenesis, {\it Adv. Cancer Res.}, \textbf{43} (1950), 175--203.

\bibitem{DN2001JMAA}
D.Gilbarg, N.S.Trudinger, Elliptic Partial Differential Equations of Second Order, {\it Springer-Verlag, Berlin/Heidelberg}, (2001).






\bibitem{Li Tao 2020 JMAA}
G. Li, Y. Tao, Analysis of a chemotaxis-convection model of capillary-sprout growth during tumor angiogenesis, {\it J. Math. Anal. Appl.}, \textbf{481} (2020), 123474.


\bibitem{H P 1997 CPDE}
H. Matthias, P. Jan, Heat kernels and maximal $L^{p}-L^{q}$ estimates for parabolic evolution equations, {\it Comm. Part. Differ. Equ.}, \textbf{22} (1997), 1647--1669.



\bibitem{OMC 1996 JMAMB}
M. Orme, M. Chaplain : A mathematical model of the first steps of tumour-related angiogenesis: Capillary sprout formation and secondary branching.  {\it IMA J. Math. Appl. Med. Biol.}, \textbf{13} (1996), 73--98.






\bibitem{PK 1989 CROH}
N. Paweletz, M. Knierim, Tumor-related angiogenesis, {\it Crit. Rev. Oncol. Hematol.}, \textbf{9} (1989), 197--242.





\bibitem{sun li 2021 JMAA}
C. Sun, Y. Li, Global bounded solution to a chemotaxis-convection model of capillary-sprout growth during tumor angiogenesis, {\it J. Math. Anal. Appl.}, \textbf{495} (2021), 124665.

\bibitem{Stokes Lauffenburger 1991 JTB}
C. Stokes, D. Lauffenburger, Analysis of the role of microvessel endothelial cell random motility and chemotaxis in angio-genesis, {\it J. Theor. Biol.}, \textbf{152} (1991), 377--403.

\bibitem{Tao Winkler 2014 JDE}
Y. Tao, M. Winkler, Energy-type estimates and global solvability in a two-dimensional chemotaxis-haptotaxis model with remodeling of non-diffusible attractant, {\it J. Differ. Equ.}, \textbf{ 257}, (2014), 784--815.





\bibitem{Tao Winkler 2015 SJMA}
Y. Tao, M. Winkler, Large time behavior in a mutidimensional chemotaxis-haptotaxis model with slow signal diffusion. {\it SIAM J. Math. Anal.}, \textbf{47} (2015), 4229--4250.


\bibitem{Tao Winkler 2021 NA}
Y. Tao, M. Winkler, The dampening role of large repulsive convection in a chemotaxis system modeling tumor angiogenesis, {\it Nonlinear Analysis}, \textbf{208} (2021), 112324.


\bibitem{Tian Zheng 2016 CPAA}
M. Tian, S. Zheng, Global boundedness versus finite-time blow-up of solutions to a quasilinear fully parabolic Keller-Segel system of two species, {\it Comm. Pure Appl. Anal.}, \textbf{15} (2016), 243--260.

\bibitem{Tang}
H. Tang, J. Zheng, K. Li, Global bounded classical solution for an attraction-repulsion chemotaxis system,  {\it Appl. Math. Letters}, \textbf{138} (2023),  108532.

\bibitem{SRGP 1985 IM}
S. Ungari, R. Katari, G. Alessandri, P. Gullino, Cooperation between fibronectin and heparin in the mobiliza-tion of capillary endothelium, {\it Invas. Metast.}, \textbf{5} (1985), 193--205.

\bibitem{WMZ2014JDE}
L. Wang, C. Mu, P. Zheng, On a quasilinear parabolic-elliptic chemotaxis system with logistic source, {\it J. Differ. Equ.}, \textbf{256} (2014), 1847--1872.

\bibitem{WWW2012SIAM}
Z. Wang, M. Winkler, D. Wrzosek, Global regularity versus infinite-time singularity formation in a chemotaxis model with volume-filling effect and degenerate diffusion, {\it SIAM J. Math. Anal.}, \textbf{44} (2012), 3502--3525.

\bibitem{Winkler2010CPDE}
M. Winkler, Boundedness in the higher-dimensional parabolic-parabolic chemotaxis system with logistic source, {\it Comm. Part. Differ. Equ.}, \textbf{35} (2010), 1516--1537.

\bibitem{wink 2013 JMPA}
M. Winkler, Finite-time blow-up in the higher-dimensional parabolic-parabolic Keller-Segel system, {\it J. Math. Pures Appl.}, \textbf{100} (2013), 748--767.

\bibitem{winkler 2010 JDE}
M. Winkler, Aggregation vs. global diffusive behavior in the higher-dimensional Keller-Segel model, {\it J. Differ. Equ.}, \textbf{248} (2010), 2889--2905.


\bibitem{winkler 2016 JMPA}
M. Winkler, Global weak solutions in a three-dimensional chemotaxis-Navier-Stokes system, {\it Ann. Inst. H. Poincar'e Anal. Non Lin\'{e}aire}, \textbf{33} (2016), 1329--1352.

\bibitem{Wu Shi 2017 CPAA}
S. Wu, J. Shi and B. Wu, Global existence of solutions to an attraction-repulsion chemotaxis model with growth, {\it Comm. Pure Appl. Anal.}, \textbf{16} (2017), 1037--1058.






\bibitem{ZJ2021JDE}
J. Xie and J. Zheng, A new result on existence of global bounded classical solution to a attraction-repulsion chemotaxis system with logistic source, {\it J. Differ. Equ.}, \textbf{298} (2021), 159--181.






\bibitem{YO 1978 Nature}
K. Yamada, K. Olden, Fibronectins-adhesive glycoproteins of cell surface and blood, {\it Nature}, \textbf{275} (1978), 179--184.






\bibitem{zhao 2020 DCDS}
J. Zhao, Large time behavior of solution to quasilinear chemotaxis system with logistic source, {\it Discrete Contin. Dyn. Syst.}, \textbf{40} (2020), 1737--1755.

\bibitem{J2018JMAA}
J. Zheng, Y. Li, G. Bao, X. Zou, A new result for global existence and  boundedness of  solutions to a parabolic-parablic Keller-Segel system with logistic source, {\it J. Math. Anal. Appl.}, \textbf{462} (2018), 1--25.

\bibitem{zheng ke 2022 CPAA}
J. Zheng, Y. Ke, Blow-up prevention by logistic source in an $n$-d chemotaxis-convection model of capillary-sprout growth during tumor angiogenesis, {\it Comm. Pure Appl. Anal.}, \textbf{22} (2023), 100--126.

\bibitem{zheng  2015 JDE}
J. Zheng, Boundedness of solutions to a quasilinear parabolic-elliptic Keller-Segel system with logistic source, {\it J. Differ. Equ.}, \textbf{259} (2015), 120--140.

\bibitem{ZZ2022JMAA}
J. Zheng, P. Zhang, Blow-up prevention by logistic source an $N$-dimensional parabolic-elliptic predator-prey system with indirect pursuit-evasion interaction, {\it J. Math. Anal. Appl.}, \textbf{519} (2022), 126741.


\bibitem{Zheng2018JMAA}
J. Zheng, Y. Li, G. Bao, X. Zou, A new result for global existence and boundedness of solutions to
a parabolic-parabolic Keller-Segel system with logistic source, {\it J. Math. Anal. Appl.}, \textbf{462} (2018), 1--25.






\end{thebibliography}
\end{document}